\documentclass[journal]{IEEEtran}

\usepackage{cite}
\usepackage{graphicx}
\usepackage{amsmath}
\usepackage{amssymb}
\usepackage{amsthm}
\usepackage{optidef}
\usepackage{algorithmic}
\usepackage{epstopdf}
\usepackage{subfig}
\usepackage{algorithm}
\usepackage{cuted}
\usepackage{array}

\newtheorem{myth}{Theorem}
\newtheorem{lemma}{Lemma}
\newtheorem{remark}{Remark}
\newtheorem{corollary}{Corollary}[myth]

\newcommand{\eg}{\textit{e}.\textit{g}.}
\newcommand{\ie}{\textit{i}.\textit{e}.}
\newcommand{\figwidth}{83mm}

\begin{document}

\title{Quickest Detection of Deception Attacks in Networked Control Systems with Physical Watermarking }

\author{Arunava Naha$^{1}$, André Teixeira$^{1}$, Anders Ahlén$^{1}$ and Subhrakanti Dey$^{2}$
	\thanks{*This work is supported by The Swedish Research Council (VR) under grants 2017-04053 and 2018-04396, and by the Swedish Foundation for Strategic Research.}
	\thanks{$^{1}$Arunava Naha, André Teixeira, and Anders Ahlén are with the Department of Electrical Engineering, Uppsala University, 751 03 Uppsala, Sweden
		{\tt\small arunava.naha@angstrom.uu.se, andre.teixeira@angstrom.uu.se, and Anders.Ahlen@angstrom.uu.se}}%
	\thanks{$^{2}$Subhrakanti Dey is with the Department of Electronic Engineering, Hamilton Institute, National University of Ireland, Maynooth, Ireland. He is also with the Department of Electrical Engineering, Uppsala University, 751 03 Uppsala, Sweden
		{\tt\small Subhra.Dey@signal.uu.se}}%
}

\markboth{This article submitted for review to IEEE Trans. Automatic Control}%
{Shell \MakeLowercase{\textit{et al.}}: Bare Demo of IEEEtran.cls for IEEE Journals}

\maketitle

\begin{abstract}
In this paper, we propose and analyze an attack detection scheme for securing the physical layer of a networked control system against attacks where the adversary replaces the true observations with stationary false data. An independent and identically distributed watermarking signal is added to the optimal linear quadratic Gaussian (LQG) control inputs, and a cumulative sum (CUSUM) test is carried out using the joint distribution of the innovation signal and the watermarking signal for quickest attack detection. We derive the expressions of the supremum of the average detection delay (SADD) for a multi-input and multi-output (MIMO) system under the optimal and sub-optimal CUSUM tests. The SADD is asymptotically inversely proportional to the expected Kullback–Leibler divergence (KLD) under certain conditions. The expressions for the MIMO case are simplified for multi-input and single-output systems and explored further to distil design insights. We provide insights into the design of an optimal watermarking signal to maximize KLD for a given fixed increase in LQG control cost when there is no attack. Furthermore, we investigate how the attacker and the control system designer can accomplish their respective objectives by changing the relative power of the attack signal and the watermarking signal. Simulations and numerical studies are carried out to validate the theoretical results.
\end{abstract}


\IEEEpeerreviewmaketitle


\section{Introduction}
\label{sec:intro}
\IEEEPARstart{L}{arge} Large distributed networked control systems (NCS) are getting deployed in various sectors such as manufacturing units, transportation systems, power systems, robotics, etc. \cite{Satchidanandan2017}. Such cyber-physical systems (CPS) consist of embedded software, processors and other physical components. The components of CPS may be distributed over a large area, and communicate with each other via wired or wireless links. Along with their innumerable advantages, there is an increasing concern regarding safety and security. In the past, there have been several incidents of attack on CPS, such as, \eg, the Stuxnet attack \cite{Langner2011}, the attack on the sewage systems in Australia \cite{Abrams2008}, the attack on the Davis-Besse nuclear power plant in Ohio, USA \cite{Alvaro1992}. Attacks on such systems can cause loss of production, financial loss, a threat to human safety, etc. Securing CPS is a great challenge. The cyber layer is usually secured by employing cryptography, digital watermarking, etc. However, these measures cannot ensure the safety of the physical layer of the system.

There are two different attack strategies such as deception attack and denial of service (DoS) attack that adversaries usually apply to attack the physical layer of CPS \cite{Mo2015}. In the deception attack, the adversary feeds the NCS with false data either by replacing or distorting the true observations and/or the control inputs \cite{Mo2015,Satchidanandan2017}. The attacker always tries to statistically match the fake data to the real ones to remain stealthy. In one scenario, the attacker records the true observations for a while and feeds the system with the recorded data along with some harmful exogenous inputs at some later point in time. Such an attack strategy is called a replay attack \cite{Mo2015}. In the DoS attack, the attacker makes the data unavailable maybe by jamming the wireless network \cite{Salimi2019}. In both the attack strategies, the attacker’s objective is to make the system unstable or force the system to operate at a state outside it's desired normal behaviour, and at the same time to remain stealthy as long as possible to cause maximum damage \cite{Mo2015,Satchidanandan2017,Salimi2019}. In this paper, we have studied mainly a specific scenario of deception attacks, where the attacker hijacks the sensor nodes and feeds random but stationary fake observations to the state estimator. The noise and the uncertainty in the system always facilitate the attacker to remain stealthy. We also assume that the attacker has complete knowledge about the system, and controller parameters and knows the statistical properties of the noise and observations. 

\subsection{Related Work}
\label{subsec:related_work}
Several different approaches are found in the literature to secure CPS from the attacks on the physical layer.  In one approach, the security of the NCS is improved by designing attack resilient state estimators which can estimate the true states with bounded errors even if there is an attack \cite{Fawzi2014, Du2019, Nicola2018}. In \cite{Park2019,Chen2018}, the authors have studied different attack strategies which will be useful to design more resilient defence strategies. The defence strategies employed for attack detections can be broadly classified into two groups, \ie, passive and active. In the passive attack detection scheme, the innovation signal is normally used as a residue signal with different statistical tests to detect attacks \cite{Pasqualetti2013, Mousavinejad2018, Ge2019}. For example, a set membership filter-based algorithm is used in \cite{Mousavinejad2018} to detect malicious data injection attacks in the NCS, a two-stage distributed deception attack detection mechanism is published in \cite{Ge2019} based on the residual analysis of the Krein state-space model and locally distributed estimators. The passive detection schemes, in general, have an unsatisfactory probability of detection in the presence of noise and uncertainties.

On the other hand, active attack detection schemes add physical watermarking signals to the control inputs to improve the probability of detection at the expense of an increased control cost \cite{Mo2009, Mo2014, Mo2015, Satchidanandan2017, Ko2019, Fang2020}. In our paper, we follow this approach to design a resilient deception attack detection scheme. The idea of physical watermarking is analogous to the digital watermarking, which is used to authenticate the actual owner of a digital content. In \cite{Mo2009}, the process of detecting a replay attack by adding a random Gaussian and independent and identically distributed (iid) watermarking signal to the linear quadratic Gaussian (LQG) control inputs is introduced. The statistics of the innovation signal changes in the presence of an attack, which is detected by a properly designed $\chi^2$ detector. In \cite{Mo2014}, the authors provide a methodology to optimise the watermarking signal power, which will maximise the detection rate for a given increase in LQG control cost. In \cite{Mo2015}, the authors further generalise the method and find the optimum watermarking signal in the class of Gaussian stationary processes by maximising a relaxed version of the Kullback–Leibler divergence (KLD) measure. In \cite{Satchidanandan2017}, the authors design two residue signals, and the time average of them will converge to some finite values when the system is under attack, otherwise, it will be zero.  It is assumed that the attacker uses a mathematical model similar to the original system to generate fake measurements, but the attacker does have any knowledge of the actual noise and the watermarking signal values. The authors have demonstrated their methodology in laboratory setup in \cite{Ko2019}. The authors consider the system model with non-Gaussian process and observation noise, and design watermarking signal for such a system in \cite{Satchidanandan2020}. In \cite{Hespanhol2018}, the authors design a statistical watermarking test to detect the attack on the sensors and the underlying communication channels. The problem of false data injection attacks in the presence of packet drop is studied in \cite{Weerakkody2018} by the design of a joint Bernoulli-Gaussian watermarking. In \cite{Fang2020}, the authors reduce the increase of control cost by designing a periodic watermarking signal. In \cite{ Pradhan2017}, the trade-off between the controller utility and the detectability of an attack is studied.

In this paper, we have studied the problem of the quickest attack detection, which has not been addressed directly in most of the reported work in the literature. The study on the topic of quickest change detection can be traced back several decades \cite{Shiryaev1963}. For our paper, we have followed the work presented in \cite{ Tartakovsky2008, Tartakovsky2017, Lai1998, Tartakovsky2014, Girardin2018}. We have taken the non-Bayesian approach of change point detection where the change point or the attack point is unknown but deterministic. In \cite{Tartakovsky2008}, it is assumed that the data before and after the change point need to be iid. We show in our study that the test data is iid before the attack, but after the attack, the test data does not remain iid. However, the test data is asymptotically stationary with or without the attack. The study in \cite{Tartakovsky2017, Lai1998, Tartakovsky2014, Girardin2018} shows that under certain conditions the cumulative sum (CUSUM) test also provides the quickest change detection, \ie, it minimises the supremum of the average detection delay (SADD) for a fixed upper limit on the average run length (ARL) for the general non-iid case. Furthermore, the SADD asymptotically converges to the inverse of the expected value of the KLD for the non-iid case provided certain conditions are satisfied \cite{Girardin2018}. We have referred to the CUSUM test using the dependent distributions for the non-iid case as the optimal CUSUM test. If the CUSUM test is performed using the non-dependent distributions for the non-iid data, then we have mentioned it as a non-optimal CUSUM test. The latter may be applicable when finding the analytic form of the dependent distributions may not be feasible.

\subsection{Motivations and Contributions}
\label{subsec:contributions}
For the safety and security of CPS, it is of paramount importance to detect the attack with minimum possible delay, thus favouring quickest sequential detection based methods. The more the attacker remains stealthy, the more damage will be caused. The watermarking based detection techniques reported in \cite{Mo2015, Satchidanandan2017, Fang2020} are not specifically designed for quickest detection of attacks. Thus we will here focus on the design and analysis of the quickest sequential detection of deception attacks by applying watermarking to the control inputs while keeping the system performance within a prescribed safety limit as recommended by the resilience requirements of CPS under attacks \cite{Cardenas2008}. We consider a linear NCS where the attacker can hijack the sensor nodes and feed fake measurement data to the estimator. The fake measurement data are assumed to be stationary and generated from a stochastic linear system. The time of the attack is unknown but deterministic in nature. The plant is controlled by a LQG controller, which receives the estimated states from a Kalman filter (KF). The controller adds a stationary but iid watermarking signal to the optimal control inputs and performs a CUSUM based test on the joint distribution of the innovation signal and the watermarking signal for the attack detection. We have reported a preliminary study on this method for the scalar case applying non-optimal CUSUM test, in \cite{Salimi2019a}. In the current paper, we extend the work significantly by considering more generalized system models, in-depth analysis of the optimal CUSUM test for the non-iid data, and extensive numerical simulations. The proposed approach can also be applied to detect a replay attack after a few modifications as reported in \cite{replay_attack}. Our main contributions are as follows.

(i) We design a sequential quickest change detection test based on the CUSUM statistics that minimises the SADD subject to a lower bound on the ARL between two consecutive false alarms. Since it is uncertain how long the system will be operational, probability of false alarm (PFA) may not be a practically useful metric \cite{Urbina2016c,Giraldo2019}. We have also shown a sub-optimal sequential detection technique which will be useful where the optimal CUSUM test may not be feasible. 

(ii) It is known that SADD is asymptotically inversely proportional to the expected KLD or the KLD between the joint stationary density of the innovation and watermarking signal with and without the attack under the optimal CUSUM or sub-optimal CUSUM test \cite{Tartakovsky2017, Girardin2018}. We derive expressions of the expected KLD for the optimal CUSUM test and KLD for the sub-optimal case. An analysis of the behaviour of the KLD with respect to the watermarking signal power and attack signal power is performed, and some structural results are presented.

(iii) We demonstrate a technique to optimise the watermarking signal variance for a multi-input and single-output (MISO) system, that maximises the expected KLD (optimal CUSUM test) or KLD (sub-optimal CUSUM test) subject to an upper bound on the increase in LQG control cost.

(iv) We take the joint distribution of the innovation signal and the watermarking signal to increase the KLD, unlike some of the previous works which consider only the innovation signal. An increase in KLD results in lower SADD, and thus in quicker detection.
\subsection{Paper Organization}
\label{subsec:organization}
The organization of the remaining part of the paper is as follows. Section~\ref{sec:system_model} describes the system model with the LQG controller and the attack strategy adopted for the paper. The mechanism of adding watermarking, the CUSUM test, and the associated detection delay are explained in Section~\ref{sec:phy_watermarking}. All the theorems and lemmas associated with multi-input and multi-output (MIMO) and MISO systems are provided in Section~\ref{sec:main_results}. The optimization technique to maximize the KLD by finding a proper watermarking signal variance is also illustrated in Section~\ref{sec:main_results}. We present numerical results in Section~\ref{sec:numerical_results} to validate the theory. Section~\ref{sec:conclusion} concludes the paper. 

\subsection{Notations}
\label{subsec:notations}
We have used capital bold letters, \eg, $\bf{A}$, $\bf{B}$, etc. to specify matrices and small bold letters, \eg, $\bf{x}$, $\bf{y}$, etc. to specify vectors, unless specified otherwise. Some special notations are given in Table~\ref{tab:notations}.
\begin{table}[h!]
	\begin{center}
		\caption{Notations}
		\label{tab:notations}
		\begin{tabular}{l|l} 
			\hline \hline
			Symbol & Description \\
			\hline
			${\rm I\!R}^{n}$ & The set of $n\times 1$ real vectors \\
			${\rm I\!R}^{m\times n}$ & The set of $m\times n$ real matrices \\
			${\bf A}^T$ & Transpose of matrix or vector ${\bf A}$ \\
			$\mathcal{N}(\mu,{\bf \Sigma})$ & Gaussian distribution with mean $\mu$ and variance $\bf \Sigma$  \\
			$\left \{ \cdot \right \}\cup \left \{\cdot \right \}$ & Union of two sets \\
			$\bf \Sigma \ge 0$  & $\bf \Sigma$ is positive semi-definite matrix \\
			$\bf \Sigma > 0$  & $\bf \Sigma$ is positive definite matrix \\
			${\bf x}_{a,k}$, ${\bf u}_{n,k}$, etc. & $k$-th instant value of the corresponding variable \\
			$[\cdot]_{ij}$ & $i$-th row and $j$-th column element of a matrix \\
			${\lambda_{\gamma ,i}}$, $\lambda_{e,i}$, etc. & $i$-th element of the corresponding vector \\
			$|\cdot|$ & Determinant of a matrix or absolute value of a scalar \\
			$tr(\cdot)$ & Trace of a matrix \\
			$\left\{{\bf X}\right\}^{k-1}_1$  & $\left\{X_i:1\le i \le k-1\right\}$\\
			\hline \hline
		\end{tabular}
	\end{center}
\end{table}



\section{System and Attack Model} \label{sec:system_model}
This section discusses the system model during the normal operations and under attack, and the attack strategy of the adversary considered in this paper. 
\subsection{System Model during Normal Operations}
\label{subsec:system_model_normal}
\begin{figure}[h!]
	\centering
	\includegraphics[width=50mm]{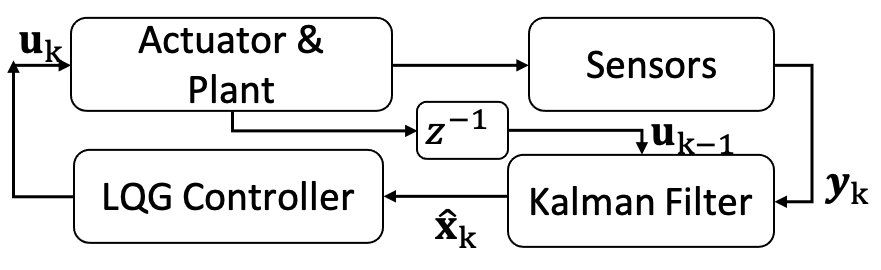}
	\caption{Schematic diagram of the system during normal operation.}
	\label{fig:sche_system_normal}
\end{figure}

We consider the following structure of the NCS, see Fig.~\ref{fig:sche_system_normal} for a schematic diagram of the complete system during the normal operation,
\begin{align}
	{\bf{x}}_{k+1}={\bf A}{\bf{x}}_{k}+{\bf B}{\bf{u}}_{k}+{\bf{w}}_{k}.
	\label{eqn:state_eqn}
\end{align}
Here ${\bf{x}}_{k}\in {\rm I\!R}^{n}$ and ${\bf{u}}_{k}\in {\rm I\!R}^{p}$ are the state and input vectors at the $k$-th time instant respectively, whereas ${\bf{w}}_{k} \in {\rm I\!R}^{n} \sim \mathcal{N}(0,{\bf Q})$ is an iid process noise. ${\bf{A}}\in {\rm I\!R}^{n\times n}$, ${\bf{B}}\in {\rm I\!R}^{n\times p}$, and ${\bf{Q}}\in {\rm I\!R}^{n\times n}$. ${\bf{Q}}\ge {\bf 0}$. Furthermore, 
\begin{align}
	{\bf{y}}_{k}={\bf C}{\bf{x}}_{k}+{\bf{v}}_{k}
	\label{eqn:obj_eqn}
\end{align}
where ${\bf{y}}_{k} \in {\rm I\!R}^{m}$ is the sensor output or the observation vector at the $k$-th time instant. Here ${\bf{C}}\in {\rm I\!R}^{m\times n}$, and ${\bf{v}}_{k} \in {\rm I\!R}^{m} \sim \mathcal{N}(0,{\bf R})$ is the iid measurement noise. We assume, ${\bf R} > {\bf 0}$. The noise vectors ${\bf{v}}_{k}$ and ${\bf{w}}_{k} $ are mutually independent, and both are independent of the initial state vector, ${\bf{x}}_{k_0}$. 
We assume the system is stabilizable and detectable. We also assume that the system has been operational for a long time, thus the system is currently at steady state. 

The Kalman filter (KF) uses the sensor measurements and the input signal information, and estimates the states as follows.
\begin{align}
	{\hat{\bf{x}}}_{k|k-1} &={\bf A}{\hat{\bf{x}}}_{k-1|k-1}+{\bf B}{\bf{u}}_{k-1} \label{eqn:est_state_update}\ \\
	{\hat{\bf{x}}}_{k|k} &={\hat{\bf{x}}}_{k|k-1}+{\bf K}\gamma_k
	\label{eqn:kf_state_eqn}
\end{align}
where ${\hat{\bf{x}}}_{k|k-1} =E[{\bf{x}}_{k}|\Psi_{k-1}]$ and ${\hat{\bf{x}}}_{k|k} =E[{\bf{x}}_{k}|\Psi_{k}]$ are the predicted and filtered state estimates respectively. $E[\cdot]$ denotes the expected value and $\Psi_{k}$ is the set of all measurements up to time $k$. The innovation $\gamma_k$ and steady state Kalman gain $\bf K$ are given by 
\begin{align}
	\gamma_k&={\bf y}_k-{\bf C}{\hat{\bf{x}}}_{k|k-1} \label{eqn:gamma_k} \ \\
	{\bf K}&={\bf P}{\bf C}^T\left({\bf C}{\bf P}{\bf C}^T+{\bf R}\right)^{-1} \label{eqn:kalman_gain}
\end{align}
where ${\bf P}=E\left[({\bf x}_k-{\hat{ \bf x}}_{k|k-1}) ({\bf x}_k-{\hat{ \bf x}}_{k|k-1})^T   \right]$ is the steady state error covariance. ${\bf P} $ is the solution to the following algebraic Riccati equation
\begin{align}
	{\bf P}={\bf A}{\bf P}{\bf A}^T+{\bf Q}-{\bf A}{\bf P}{\bf C}^T\left( {\bf C}{\bf P}{\bf C}^T +{\bf R} \right)^{-1}{\bf C}{\bf P}{\bf A}^T.
	\label{eqn:P}
\end{align}

The control input ${\bf u}_k$ is generated by minimizing the following infinite horizon LQG cost 
\begin{align}
	J=\lim_{T\to\infty}E\left[ \frac{1}{2T+1}\left\{ \sum_{k=-T}^T \left( {\bf x}^T_k{\bf W}{\bf x}_k+{\bf u}^T_k{\bf U}{\bf u}_k\right)\right\} \right]
	\label{eqn:cost_fun_J}
\end{align}
where ${\bf{W}}\in {\rm I\!R}^{n\times n}$ and ${\bf{U}}\in {\rm I\!R}^{p\times p}$ are positive definite diagonal weight matrices. The optimum input appears as a fixed gain linear control signal given by
\begin{align}
	{\bf u}^*_k&={\bf L}{\hat {\bf x}}_{k|k}  \label{eqn:opt_u} \ \\
	{\bf L}&=-\left( {\bf B}^T{\bf S}{\bf B}+{\bf U}\right)^{-1}{\bf B}^T{\bf S}{\bf A} \label{eqn:L}
\end{align}
where $\bf S$ is the solution to the following algebraic Riccati equation,
\begin{align}
	{\bf S}={\bf A}^T{\bf S}{\bf A}+{\bf W}-{\bf A}^T{\bf S}{\bf B}\left({\bf B}^T{\bf S}{\bf B}+{\bf U}\right)^{-1}{\bf B}^T{\bf S}{\bf A}.
	\label{eqn:S}
\end{align}

\subsection{Attack Strategy and Changes in System Model}
\label{subsec:attack_strategy}
The attack strategy of the adversary considered in this paper is discussed here. We assume that the attacker has the following knowledge about the system.  
\begin{enumerate}
	\item The attacker knows the system parameters ${\bf A}$, ${\bf B}$, ${\bf C}$, ${\bf Q}$, and ${\bf R}$, and also the control policy, \ie, ${\bf L}$.
	\item The attacker can tamper with the integrity of the sensor nodes and feed undesired information to the system.
	\item The attacker does not have access to the control signal or the controller.  
\end{enumerate}
The objective of the adversary is to cause harm to the system by replacing the true sensor measurements ${\bf y}_k$ by fake observations ${\bf z}_k$, and at the same time remain stealthy.  The adversary can achieve his goal by jamming or overpowering the true sensor data sent over a wireless link or by hijacking the sensor nodes (man-in-the-middle attack). The adversary will also try to remain undetected as long as possible to cause maximum damage to the system. Figure~\ref{fig:sche_system_attack} shows a schematic diagram of the system under attack. The system is assumed to be normal till the time $k < \nu$, and the attacker replaces the true observation ${\bf y}_k$ by the fake observation ${\bf z}_k$ at a deterministic but unknown time instant $k=\nu$, and keeps on injecting the fake observation for $k\ge \nu$. 
\begin{figure}[h!]
	\centering
	\includegraphics[width=65mm]{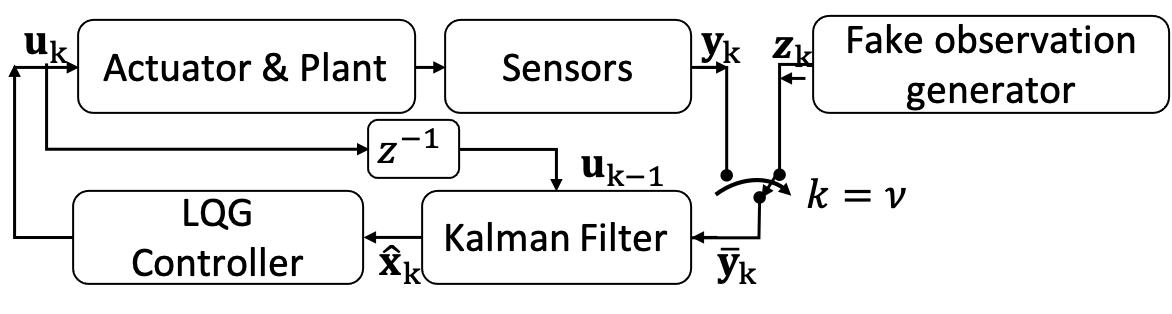}
	\caption{Schematic diagram of the system under attack. $\bar {\bf y}_k={\bf y}_k \text{ if } k<\nu, \bar {\bf y}_k={\bf z}_k \text{ otherwise}$.}
	\label{fig:sche_system_attack}
\end{figure}
It is assumed that the fake observations will be generated by the following stochastic linear system
\begin{align}
	{\bf z}_{k}&={\bf A}_a{\bf z}_{k-1}+{\bf w}_{a,k-1} \label{eqn:hidden_states_main}
\end{align}
where ${\bf z}_{k} \in {\rm I\!R}^{m}$, and ${\bf w}_{a,k} \sim \mathcal{N}(0,{\bf Q}_a)$ is the iid noise vector at the $k$-th time instant. ${\bf Q}_a \in {\rm I\!R}^{m \times m}$ and ${\bf Q}_a\ge 0$. The attacker will try to keep the statistical properties of ${\bf z}_k$, i.e., mean and variance, similar to the true observation ${\bf y}_k$ to remain stealthy. Since the true measurement ${\bf y}_k$ is stationary, the attacker will keep the fake measurement ${\bf z}_k$ stationary by taking the initial covariance of ${\bf z}_{k}$ as ${\bf E}_{zz}(0)\triangleq E\left[\bf{z}_k\bf{z}_k^T\right]$ to remain stealthy, where ${\bf E}_{zz}(0)$ is the solution to the following Lyapunov equation,
\begin{align}
	{\bf E}_{zz}(0)={\bf A}_a{\bf E}_{zz}(0){\bf A}^T_a+{\bf Q}_a.
	\label{eqn:czo}
\end{align}

The estimated states from the Kalman filter will take the following form when the system is under attack, \ie, $k\ge \nu$,
\begin{align}
	{\hat{\bf{x}}}^F_{k|k-1} &={\bf A}{\hat{\bf{x}}}^F_{k-1|k-1}+{\bf B}{\bf{u}}_{k-1} \label{eqn:xF_k_k_1}\ \\
	{\hat{\bf{x}}}^F_{k|k} &={\hat{\bf{x}}}^F_{k|k-1}+{\bf K}{\widetilde \gamma}_k  \label{eqn:kf_state_eqn_attack}\ \\
	{\widetilde {\bf \gamma}}_k &= {\bf z}_{k} - {\bf C}{\hat{\bf{x}}}^F_{k|k-1} \label{eqn:gamma_k_attack}.
\end{align}
It is the same Kalman filter as given in (\ref{eqn:est_state_update})-(\ref{eqn:P}) with the true observation ${\bf y}_k$ replaced by the fake data ${\bf z}_k$. So, the defender does not need to change anything for the Kalman filter during the attack. 

An attacker can make the system unstable by following the described attack model. For illustration, the true and estimated states of System-A is plotted in Fig.~\ref{fig:true_est_states_unstable} when the system is under attack from the time instant $k=500$. See the model parameters of System-A from Appendix~\ref{apdx:system_param}. The system becomes unstable soon after the attack. 
\begin{figure}[h!]
	\centering
	\includegraphics[width=\figwidth]{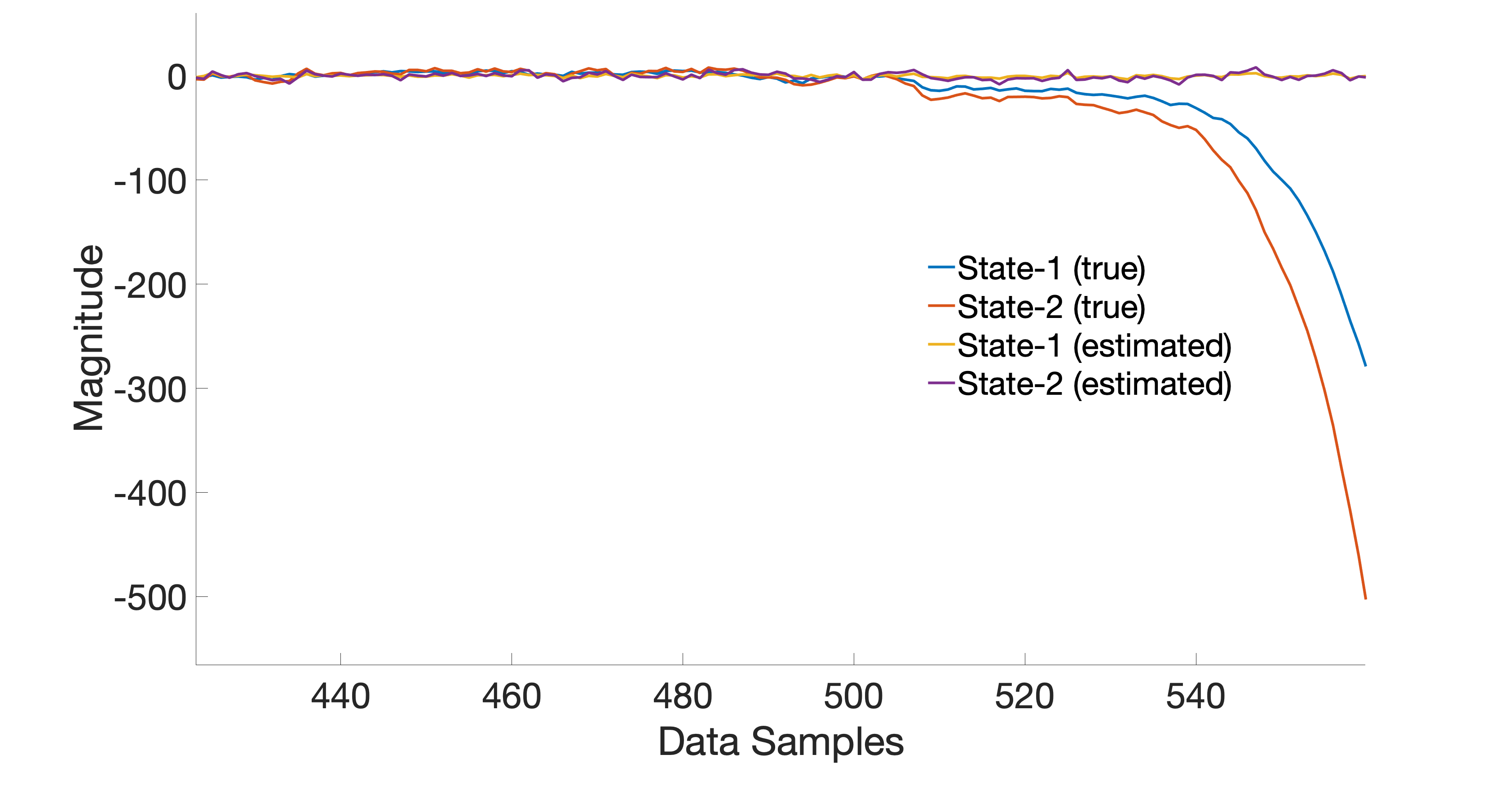}
	\caption{True and estimated states of System-A.}
	\label{fig:true_est_states_unstable}
\end{figure}



\section{Physical Watermarking based Defence Mechanism and Delay in Detection}
\label{sec:phy_watermarking}
This section proposes the physical-watermarking-based sequential attack detection scheme and discusses about the delay in the detection process. We use hypothesis testing to detect the attack. There are two different hypotheses to choose from, 
\begin{itemize}
	\item $H_0$: No attack. Estimator receives the true observation ${\bf y}_k$ 
	\item $H_1$: Attack. Estimator receives a fake observation ${\bf z}_k$. 
\end{itemize}
The innovation signals ((\ref{eqn:gamma_k_attack}) and (\ref{eqn:gamma_k})) under attack and no attack contain different information. Therefore, the innovation signal is the natural selection of information source for hypothesis testing. The probability density functions (PDF) of ${\gamma_k}$ and $\widetilde{\gamma}_k$ are denoted as $f_{\gamma_k}({\bar \gamma}_k)$ and $f_{{\widetilde \gamma}_k}({\bar \gamma}_k)$ respectively, where ${\bar \gamma}_k={ \gamma}_k$  before attack, and ${\bar \gamma}_k={\widetilde \gamma}_k$ after attack. Both the distributions $f_{\gamma_k}({\bar \gamma}_k)$ and $f_{{\widetilde \gamma}_k}({\bar \gamma}_k)$ are stationary in nature. The probability of attack detection will increase if the KLD i.e., $D\left(f_{\widetilde \gamma_k},f_{\gamma_k}\right)$, between the two distributions $f_{\widetilde \gamma_k}({\bar \gamma}_k)$ and $f_{ \gamma_k}({\bar \gamma}_k)$ under $H_1$ and $H_0$ increases \cite{Tartakovsky2014},
\begin{align}
	D\left(f_{\widetilde \gamma_k},f_{\gamma_k}\right)=\int_{ {\rm I\!R}^{m}} f_{\widetilde \gamma_k}({\bar \gamma})\log\frac{f_{\widetilde \gamma_k}({\bar \gamma})}{f_{\gamma_k}({\bar \gamma})}d{\bar \gamma}.
	\label{eqn:eqn_d}
\end{align}
The adversary will always try to remain stealthy by keeping the KLD low and thus cause maximum damage to the system. Therefore, the task of the control system designer is to maximize the KLD, thus making it difficult for the attacker to remain stealthy. Disturbances and measurement noise create uncertainty which favours the adversary. 

\subsection{Physical Watermarking}
A well-adopted technique to detect attacks on the control system is to add a watermarking signal, as described above \cite{Mo2015,Satchidanandan2017}. The control designer thus adds a random watermarking signal ${\bf e}_k$ to the optimal LQG control input ${\bf u}_k^*$, see (\ref{eqn:add_ek}). The actual values of the watermarking signal will only be known to the controller and not to the attacker. However, the attacker may know the statistics of the watermarking signal. 
\begin{align}
	{\bf u}_k = {\bf u}^*_k+{\bf e}_k
	\label{eqn:add_ek}
\end{align}
where ${\bf u}^*_k$ is the optimal input (\ref{eqn:opt_u}), ${\bf e}_k \sim \mathcal{N}(0,{\bf \Sigma}_e)$ is an iid process, and ${\bf \Sigma}_e\ge 0$, and possibly non-diagonal matrix. In the literature, ${\bf e}_k$ is also taken to be a stationary Gauss-Markov process by some researchers. However, for our work, we assume it to be iid. The addition of ${\bf e}_k$ provides a means to the controller to check the authenticity of the measurement signal fed to the system. The distribution of the innovation signal will change substantially if the true measurement ${\bf y}_k$, which is correlated to ${\bf e}_{k-1}$, is replaced by ${\bf z}_k$, which is independent of ${\bf e}_{k-1}$, even if the attacker knows the statistics of ${\bf e}_{k}$.

Detection of the attack as early as possible is of utmost importance to reduce the damage. The optimal Neyman-Pearson (NP) test \cite{Mo2015} and the asymptotic test \cite{Satchidanandan2017} reported in the literature for the attack detection do not address the challenge of earliest detection. To this end, we have adopted a non-Bayesian sequential detection scheme \cite{Tartakovsky2014} to detect the attack at the earliest time instant. It is assumed the attack takes place at a deterministic but unknown point in time. Instead of using the innovation signals $\gamma_k$ and ${\widetilde {\bf \gamma}}_k$ alone, we use the joint distributions of $\gamma_k$ and ${\bf e}_{k-1}$, and ${\widetilde {\bf \gamma}}_k$ and ${\bf e}_{k-1}$ for the test. We show the simulation results in the Section~\ref{sec:numerical_results} that such a choice reduces the detection delay. 
The innovation signal during normal operation of the system and under attack will take the following forms (\ref{eqn:gamma_e}) and (\ref{eqn:gamma_e_attack}), respectively,
\begin{align}
	\gamma_k&={\bf y}_k-{\bf C}{\bf \hat x}_{k|k-1}  \cr
	&={\bf C}{\bf A}\left({\bf x}_{k-1}-{\bf \hat x}_{k-1|k-1}\right)+{\bf C}{\bf w}_{k-1}+{\bf v}_k , \label{eqn:gamma_e} \ \\
	\widetilde\gamma_k&={\bf z}_k-{\bf C}{\bf \hat x}^F_{k|k-1} \cr
	&={\bf z}_k-{\bf C}\left({\bf A}+{\bf B}{\bf L}\right){\bf \hat x}^F_{k-1|k-1}-{\bf C}{\bf B}{\bf e}_{k-1}.
	\label{eqn:gamma_e_attack}
\end{align}
It is evident from (\ref{eqn:gamma_e}) and (\ref{eqn:gamma_e_attack}) that the innovation signal during the normal operation of the system will be uncorrelated with the watermarking signal. However, on the contrary, the innovation signal will be correlated with the watermarking signal during the attack. 

\subsection{Detection Delay} 
\label{subsec:detection_delay}
We use the delay in the attack detection as the metric to measure the performance of the defence strategy. Here we adopt the theory of asymptotic optimality of the CUSUM test when the signal before and after the change (attack) may not be iid \cite{Tartakovsky2014}. We start this section by introducing the definitions of relevant terms as follows. \ \\
\textbf{Average Detection Delay (ADD)}: ADD is defined as 
\begin{equation}
	ADD \triangleq E_{\nu}\left[ T_{H_1}-\nu|T_{H_1}>\nu\right]
	\label{eqn:add}
\end{equation}
where $E_{\nu}[\cdot]$ is the expectation taken with respect to the PDF under attack. Here $\nu$ is the attack starting point in time which is assumed to be unknown but deterministic in nature, whereas $T_{H_1}$ is the attack starting point detected by a hypothesis testing algorithm.  \ \\
\textbf{Supremum Average Detection Delay (SADD)}: SADD is defined as
\begin{equation}
	SADD\triangleq \sup _{1\le\nu<\infty} E_{\nu}\left[ T_{H_1}-\nu|T_{H_1}>\nu\right].
	\label{eqn:sadd}
\end{equation}
\textbf{Average Run Length (ARL)}: ARL is defined as
\begin{equation}
	ARL\triangleq E_{\infty}\left[T_{H_1}\right]
	\label{eqn:arl}
\end{equation}
where $E_{\infty}[\cdot]$ is the expectation taken with respect to the PDF when there is no attack, \ie, $\nu=\infty$. ARL represents the average time between two false alarms.  \ \\
Ideally, we would like to have a detection scheme that will minimize ADD for any value of $\nu$ for a fixed threshold on ARL. However, such a detection scheme does not exist \cite{Tartakovsky2014}. We can only find a procedure that will minimize the worst-case ADD for any $\nu$, \ie, SADD, for a fixed threshold on ARL. As per the theory presented in \cite{Tartakovsky2014}, CUSUM is one of such procedures. The CUSUM procedure is asymptotically minimax in the sense of minimizing the SADD for all $\nu > 0$, as $ARL_h \rightarrow \infty$, and the minimum SADD is   
\begin{equation}
	SADD \sim \frac{\log(ARL_h)}{I} 
	\label{eqn:add_lim}
\end{equation}
where $I$ is a finite positive real number, $ARL_h$ is the threshold on ARL, $ARL\ge ARL_h$, provided the following three conditions are satisfied \cite{Tartakovsky2014}:  
\begin{align}
	&\text{i) }\frac{1}{n}\lambda^\nu_{\nu+n}\xrightarrow[n\rightarrow\infty]{P_\nu}I  \label{eqn:condn_1}, \  \\
	&\text{ii) }\sup_{0\le\nu<\infty} \ ess  \sup P_\nu\left\{M^{-1}\max_{0\le n<M}\lambda^\nu_{\nu+n}\ge \right. \cr
	&\left. (1+\epsilon)I|\Psi_\nu\right\} \xrightarrow[M\rightarrow\infty]{}0,\ \forall\ \epsilon >0  \label{eqn:condn_2} \text{, and } \ \\
	&\text{iii) }\sup_{0\le\nu<k} \ ess  \sup P_\nu\left\{n^{-1}\lambda^k_{k+n}<I(1-\epsilon)|\Psi_\nu \right\}\xrightarrow[n\rightarrow\infty]{}0, \cr
	&\   \forall\ 0<\epsilon<1\ \text{ and }\ k\ge0   \label{eqn:condn_3} 
\end{align}
where $P_\nu$ indicates the probability after the change and $M$ is a positive integer variable. Here $\Psi_\nu$ is the set of all observations up until the change point $\nu$. The variable $\lambda^\nu_{\nu+n}$ is defined as
\begin{equation}
	\lambda^\nu_{\nu+n} \triangleq \sum_{k=\nu+1}^{n+\nu}\log\frac{f_{\nu,k}\left( X_k|\left\{{\bf X}\right\}^{k-1}_1\right)}{f_{\infty,k}\left( X_k|\left\{{\bf X}\right\}^{k-1}_1\right)}
	\label{eqn:lambda}
\end{equation}
where $X_k$ is the observation at the $k$-th time instant and $\left\{{\bf X}\right\}^{k-1}_1=\left\{X_i:1\le i \le k-1\right\}$. In (\ref{eqn:lambda}), $f_{\nu,k}(\cdot|\cdot)$ and $f_{\infty,k}(\cdot|\cdot)$  are the PDFs of the observations at the $k$-th time instant for an attack starting at $\nu$ and without an attack, respectively.

For the case of attack detection using the joint distributions of innovation and watermarking signals,
\begin{equation}
	\lambda^\nu_{\nu+n}=\sum_{k=\nu+1}^{n+\nu}\log\frac{f_{{\widetilde {\bf \gamma}_k},{\bf e}_{k-1}}\left({{ {\widetilde {\bf \gamma}}}_k,{\bf e}_{k-1}}|\left\{\bar\gamma\right\}_1^{k-1},\left\{{\bf e}\right\}_1^{k-2}\right)}{f_{{ {\bf \gamma}_k},{\bf e}_{k-1}}\left({{ \widetilde {\bf \gamma}}_k,{\bf e}_{k-1}}|\left\{\bar\gamma\right\}_1^{k-1},\left\{{\bf e}\right\}_1^{k-2}\right)}
	\label{eqn:lambda_attack_case}
\end{equation} 
where $f_{{\widetilde {\bf \gamma}_k},{\bf e}_{k-1}}(\cdot|\cdot)$ and $f_{ {{\bf \gamma}_k},{\bf e}_{k-1}}(\cdot|\cdot)$ are the joint dependent distributions of the innovation signal at the $k$-th time instant and watermarking signal at $(k-1)$-th time instant for the attack and no attack cases, respectively. $\left\{\bar\gamma\right\}_1^{k-1} = \left\{\gamma_i:1\le i < \nu\right\}\cup\left\{{\widetilde \gamma}_i:\nu\le i \le k-1\right\}$.
The data ($\gamma_k$, $\widetilde{\gamma}_k$ and ${\bf e}_{k-1}$) satisfy the mean ergodicity theorem because of their stationarity property. The previously mentioned three conditions are satisfied under the mean ergodicity property of the data, and we can say $I$ converges to the expected value of the KLD between $f_{{\widetilde {\bf \gamma}_k},{\bf e}_{k-1}}(\cdot|\cdot)$ and $f_{{\bf \gamma}_k,{\bf e}_{k-1}}(\cdot|\cdot)$ as $n \rightarrow \infty$ \cite{Girardin2018}. In other words, 
\begin{align}
	&I\rightarrow \frac{1}{n}\sum_{k=\nu+1}^{n+\nu}\log\frac{f_{{\widetilde {\bf \gamma}_k},{\bf e}_{k-1}}\left({{\widetilde {\bf \gamma}}_k,{\bf e}_{k-1}}|\left\{\bar\gamma\right\}_1^{k-1},\left\{{\bf e}\right\}_1^{k-2}\right)}{f_{{ {\bf \gamma}_k},{\bf e}_{k-1}}\left({{ \widetilde {\bf \gamma}}_k,{\bf e}_{k-1}}|\left\{\bar\gamma\right\}_1^{k-1},\left\{{\bf e}\right\}_1^{k-2}\right)},\nonumber \\
	&\text{as } n\rightarrow\infty \text{, which converges to the following form,} \cr
	&E\left[ \int_{{\rm I\!R}^{m+p}} \log\frac{f_{{\widetilde {\bf \gamma}_k},{\bf e}_{k-1}}\left({{\widetilde {\bf \gamma}}_k,{\bf e}_{k-1}}|\left\{\bar\gamma\right\}_1^{k-1},\left\{{\bf e}\right\}_1^{k-2}\right)}{f_{{ {\bf \gamma}_k},{\bf e}_{k-1}}\left({{\widetilde {\bf \gamma}}_k,{\bf e}_{k-1}}|\left\{\bar\gamma\right\}_1^{k-1},\left\{{\bf e}\right\}_1^{k-2}\right)}\right. \cr 
	&\left. f_{{\widetilde {\bf \gamma}_k},{\bf e}_{k-1}}\left({{\widetilde {\bf \gamma}}_k,{\bf e}_{k-1}}|\left\{\bar\gamma\right\}_1^{k-1},\left\{{\bf e}\right\}_1^{k-2}\right)d{\bf \gamma}d{\bf e} \right]  \cr
	&=E\left[ D\left(f_{{\widetilde {\bf \gamma}_k},{\bf e}_{k-1}},f_{{ {\bf \gamma}_k},{\bf e}_{k-1}} |\left\{\bar\gamma\right\}_1^{k-1},\left\{{\bf e}\right\}_1^{k-2}\right)\right]. \label{eqn:I_eq_kld} 
\end{align}
Here, the expectation is taken over the joint  distribution of  $\left\{\bar\gamma\right\}_1^{k-1},\left\{{\bf e}\right\}_1^{k-2}$. 

\subsection{Optimal and Sub-optimal CUSUM Tests}
\label{eqn:cusum_test}
The following CUSUM test will minimize the SADD asymptotically,
\begin{align}
	&gd_k= \nonumber \\
	&\max\left(0,gd_{k-1}+\log\frac{f_{{\widetilde {\bf \gamma}_k},{\bf e}_{k-1}}\left({{\bar {\bf \gamma}}_k,{\bf e}_{k-1}}|\left\{\bar\gamma\right\}_1^{k-1},\left\{{\bf e}\right\}_1^{k-2}\right)}{f_{{ {\bf \gamma_k}},{\bf e}_{k-1}}\left({{\bar {\bf \gamma}}_k,{\bf e}_{k-1}}\right)}\right)
	\label{eqn:cusum_opt} \\
	&\text{where } {\bar \gamma}_k={ \gamma}_k  \text{ before attack, and }  {\bar \gamma}_k={\widetilde \gamma}_k \text{ after attack, and} \nonumber \\
	&SADD^*
	\rightarrow \frac{\log(ARL_h)}{E\left[D\left(f_{{\widetilde {\bf \gamma}_k},{\bf e}_{k-1}},f_{{ {\bf \gamma}_k},{\bf e}_{k-1}}|\left\{\bar\gamma\right\}_1^{k-1},\left\{{\bf e}\right\}_1^{k-2} \right)\right]}, \nonumber \\ 
	&\ as\ ARL_h \rightarrow \infty.
	\label{ref:opt_SADD}
\end{align}
Since before the attack the innovation signal ${\bf \gamma}_k$ and the watermarking signal ${\bf e}_{k-1}$ both are iids, and also uncorrelated to each other, the non-dependent distribution is used in the denominator of (\ref{eqn:cusum_opt}). The controller decides on hypothesis $H_0$ or $H_1$ based on the following test,
\begin{description}
	\item[$H_0:$] Selected, when $gd_k < \log(ARL_h)$
	\item[$H_1:$] Selected, when $gd_k \ge  \log(ARL_h)$.
\end{description}
For certain cases, the closed-form expressions for the dependent distributions may not be found analytically, or it may be computationally too complex. Under such scenarios, the following sub-optimal CUSUM test can be carried out using the non-dependent distributions for sequential attack detection,
\begin{equation}
	g_k=\max\left(0,g_{k-1}+\log\frac{f_{{\widetilde {\bf \gamma}_k},{\bf e}_{k-1}}\left({{\bar {\bf \gamma}}_k,{\bf e}_{k-1}}\right)}{f_{{ {\bf \gamma_k}},{\bf e}_{k-1}}\left({{\bar {\bf \gamma}}_k,{\bf e}_{k-1}}\right)}\right).
	\label{eqn:cusum_subopt}
\end{equation}
Under the assumption that the system has been operating under a sufficiently long time, the joint distributions of the innovation and watermarking signal  converge to their stationary distributions. Therefore, in what follows, we use only the stationary PDFs for the sub-optimal case. Under the sub-optimal CUSUM test, the SADD will converge as follows, since $I$ (\ref{eqn:add_lim}) converges to ${D\left(f_{{\widetilde {\bf \gamma}_k},{\bf e}_{k-1}},f_{{ {\bf \gamma}_k},{\bf e}_{k-1}}\right)}$.
\begin{align}
	&SADD
	\rightarrow \frac{\log(ARL_h)}{D\left(f_{{\widetilde {\bf \gamma}_k},{\bf e}_{k-1}},f_{{ {\bf \gamma}_k},{\bf e}_{k-1}}\right)}, \ as\ ARL_h \rightarrow \infty.
	\label{ref:subopt_SADD}
\end{align}
The test statistics $g_k$ is compared with the threshold $\log(ARL_h)$ as before.



\section{Main Results}
\label{sec:main_results}
We derive the expressions of the probability distributions, KLD and $\Delta LQG$ to evaluate the performance of the proposed detector analytically. We first state the theorems for the general MIMO systems in Sub-section~\ref{subsec:mimo_system}, and then simplify the theorems for the MISO systems in Subsection~\ref{subsec:miso_system} to acquire better structural understanding. The technique to optimize the ${\bf \Sigma}_e$ to achieve minimum SADD for a given upper bound on the $\Delta LQG$ is illustrated in Subsection~\ref{subsec:opt_e}.
\subsection{Multiple Input Multiple Output Systems}
\label{subsec:mimo_system}
\begin{myth}\label{th:opt_cusum}
	The optimal CUSUM test to detect the deception attack given by (\ref{eqn:hidden_states_main}) will take the following form,
	\begin{align}
		&gd_k=\max\left(0,gd_{k-1}+\log\frac{f_{{\widetilde {\bf \gamma}_k}}\left({{ \bar {\bf \gamma}}_k}|\left\{\bar \gamma\right\}_1^{k-1},\left\{{\bf e}\right\}_1^{k-1}\right)}{f_{{ {\bf \gamma_k}}}\left({{\bar {\bf \gamma}}_k}\right)}\right),
		\label{eqn:th_cusum_opt} \\
		&\text{where } {\bar \gamma}_k={ \gamma}_k  \text{ before attack, and }  {\bar \gamma}_k={\widetilde \gamma}_k \text{ after attack}, \nonumber
	\end{align}
	\begin{align}
		&\left\{{{\widetilde {\bf \gamma}}_k}|\left\{\bar \gamma\right\}_1^{k-1},\left\{{\bf e}\right\}_1^{k-1}\right\} \nonumber \\
		& \sim \mathcal{N}\left({\bf \mu}_{{\widetilde {\bf \gamma}_k}|\left\{\bar \gamma\right\}_1^{k-1},\left\{{\bf e}\right\}_1^{k-1}},{\bf \Sigma}_{{\widetilde {\bf \gamma}_k}|\left\{\bar \gamma\right\}_1^{k-1},\left\{{\bf e}\right\}_1^{k-1}}\right), \cr
		&{\bf \mu}_{{\widetilde {\bf \gamma}_k}|\left\{\bar \gamma\right\}_1^{k-1},\left\{{\bf e}\right\}_1^{k-1}} =  \cr 
		& \begin{cases}{\bf A}_a{\bf z}_{k-1}-{\bf C}\left({\bf A}+ {\bf B}{\bf L} \right){\bf {\hat x}}^F_{k-1|k-1} - {\bf C}{\bf B}{\bf e}_{k-1}, & k \ge \nu \\{\bf A}_a{\bf y}_{k-1}-{\bf C}\left({\bf A}+ {\bf B}{\bf L} \right){\bf {\hat x}}_{k-1|k-1} - {\bf C}{\bf B}{\bf e}_{k-1}, & k < \nu\end{cases} \label{eqn:f_mean_attack_dept} \\
		& {\bf \Sigma}_{{\widetilde {\bf \gamma}_k}|\left\{\bar \gamma\right\}_1^{k-1},\left\{{\bf e}\right\}_1^{k-1}}={\bf Q}_a, \label{eqn:f_gamma_attack_dept} \text{ and } \\
		& { {\bf \gamma}}_k \sim\mathcal{N}\left({\bf 0},{\bf \Sigma}_{{ {\bf \gamma}}} \right), \nonumber \\
		& {\bf \Sigma}_{{\bf \gamma}}={\bf C}{\bf P}{\bf C}^T+{\bf R}.
		\label{eqn:f_gamma_normal}
	\end{align}
\end{myth}
\begin{proof}[Proof]
	The proof of Theorem~\ref{th:opt_cusum} is provided in Appendix~\ref{apdx:opt_cusum}.
\end{proof}
\begin{remark}
	The likelihood ratio in (\ref{eqn:th_cusum_opt}) will be evaluated using the innovation signal ${\bar \gamma}_k$ from the Kalman filter. ${\bar \gamma}_k = { \gamma}_k$ if $k< \nu$, and it will change automatically to ${\bar \gamma}_k = { \widetilde \gamma}_k$ if $k \ge \nu$ without any intervension from the defender. Similarly, ${\bf y}_k$ and ${\bf {\hat x}}_{k-1|k-1}$ will change to ${\bf z}_k$ and ${\bf {\hat x}^F}_{k-1|k-1}$, respectively, after the attack, as given in (\ref{eqn:f_mean_attack_dept}). However, the attacker plays an active role by replacing the true observation ${\bf y}_k$ by the fake data ${\bf z}_k$ at $k \ge \nu$. 
	%
\end{remark}
\begin{remark}
	The optimal CUSUM test utilising the dependent distributions of the innovation signals before and after an attack is performed employing Theorem~\ref{th:opt_cusum}. The innovation signal $\gamma_k$ before an attack is iid, and uncorrelated to the watermarking signal $\bf{e}_{k-1}$. Therefore, the non-dependent distribution is used in (\ref{eqn:th_cusum_opt}) for $\gamma_k$. On the other hand, the innovation signal after an attack ${\widetilde \gamma}_k$ is dependent on its previous values and watermarking signal values. Therefore, the dependent distribution of ${\widetilde \gamma}_k$ is used in (\ref{eqn:th_cusum_opt}), and the derived dependent mean and covariance are given in (\ref{eqn:f_mean_attack_dept})-(\ref{eqn:f_gamma_attack_dept}). The dependent variance is fixed. However, the dependent mean is changing for every time step depending on the previous measurement, estimated state and watermarking signal values. 
\end{remark}
\begin{corollary} \label{cor:subopt_cusum}
	The sub-optimal CUSUM test using the non-conditional distributions to detect the deception attack given by (\ref{eqn:hidden_states_main}) will take the following form,
	\begin{align}
		&g_k=\max\left(0,g_{k-1}+\log\frac{f_{{\widetilde {\bf \gamma}_k},{\bf e}_{k-1}}\left({{\bar {\bf \gamma}}_k,{\bf e}_{k-1}}\right)}{f_{{ {\bf \gamma_k}},{\bf e}_{k-1}}\left({{\bar {\bf \gamma}}_k,{\bf e}_{k-1}}\right)}\right),
		\label{eqn:corr_cusum_subopt} \\
		&\text{where }{\bar \gamma}_k={ \gamma}_k  \text{ before attack, and }  {\bar \gamma}_k={\widetilde \gamma}_k \text{ after attack}, \cr
		& {\bf \gamma}_{e,k}=\left[{\bf \gamma}_k^T, {\bf e}_{k-1}^T \right]^T  \sim{\cal{N}}\left( {\bf 0},{\bf \Sigma}_{{\gamma}_e} \right), \nonumber \\
		&\text{where } {\bf \Sigma}_{\gamma_e}=\begin{bmatrix}{\bf \Sigma}_\gamma & {\bf 0}_{m\times p} \\{\bf 0}_{p\times m} & {\bf \Sigma}_{e} \end{bmatrix}, \label{eqn:sgima_sq_gamma_e} \text{ and } \\
		& {\bf \widetilde \gamma}_{e,k}=\left[{\bf \widetilde \gamma}_k^T, {\bf e}_{k-1}^T \right]^T \sim{\cal{N}}\left( {\bf 0},{{\bf \Sigma}_{{\widetilde \gamma}_e}}\right), \nonumber
	\end{align}
	\begin{align}	
		&\text{where } {\bf \Sigma}_{\widetilde \gamma_e}=\begin{bmatrix}{\bf \Sigma}_{\widetilde\gamma} & - {\bf C} {\bf B} {\bf \Sigma}_e \\ - {\bf \Sigma}_e{\bf B}^T {\bf C}^T   & {\bf \Sigma}_{e} \end{bmatrix} \label{eqn:sgima_sq_gamma_e_attack}.
	\end{align}
\end{corollary}
\begin{proof}[Proof]
	The proof of Corollary~\ref{cor:subopt_cusum} is provided in Appendix~\ref{apdx:subopt_cusum}.
\end{proof}
\begin{remark}
	Both the test statistics $gd_k$ and $g_k$ will be close to zero during the normal operation, and they will gradually increase after the attack at every time step. 
\end{remark}
\begin{remark}
	For the sub-optimal CUSUM test, the non-dependent and asymptotically stationary distributions of ${\gamma}_k$ and ${\widetilde \gamma}_k$ are used. Such a test can be applied when designing the optimal CUSUM test is not feasible, \eg, replay attack detection as discussed in \cite{replay_attack}. Also, for the optimal CUSUM test, the dependent mean needs to be evaluated at every time step, which increases the computational complexity compared to the sub-optimal CUSUM test. 
\end{remark}

%
%
%
%
%

\begin{lemma}\label{lemma:sigma_gammas}
	The covariance matrix ${\bf \Sigma}_{\widetilde \gamma}$ of the innovation signal ${\widetilde \gamma}$ after the attack  will take the following form,
	\begin{align}
		{\bf \Sigma}_{\widetilde \gamma}&={\bf E}_{zz}(0)-{\bf C}({\bf A}+{\bf B{\bf L)}}{\bf E}_{xz}(-1) \cr
		&-\left[{\bf C}({\bf A}+{\bf B{\bf L)}}{\bf E}_{xz}(-1) \right]^T+{\bf C}{\bf B}{\bf \Sigma}_e{\bf B}^T{\bf C}^T  \cr
		&+{\bf C}({\bf A}+{\bf B}{\bf L}){\bf \Sigma}_{x^Fz}({\bf A}+{\bf B}{\bf L})^T{\bf C}^T \cr
		&+{\bf C}({\bf A}+{\bf B}{\bf L}){\bf \Sigma}_{x^Fe}({\bf A}+{\bf B}{\bf L})^T{\bf C}^T,
		\label{eqn:sigma_gamma_attack} \\
		&\text{where }{\bf E}_{xz}(-1)= \sum_{i=0}^\infty\mathcal{A}^{i}{\bf K}{\bf A}_a^{i+1}{\bf E}_{zz}\left(0\right) \label{eqn:Exz1_original}
	\end{align}
	and ${\bf E}_{zz}(0) = E\left[{\bf z}_k {\bf z}_k^T\right]$. ${\bf \Sigma}_{x^Fz}$ and ${\bf \Sigma}_{x^Fe}$ are the solutions to the following Lyapunov equations,
	\begin{align}
		&\mathcal{A}{\bf \Sigma}_{x^Fz}\mathcal{A}^T-{\bf \Sigma}_{x^Fz}+{\bf K}{\bf E}_{zz}(0){\bf K}^T+\mathcal{A}{\bf E}_{xz}(-1){\bf K}^T \cr
		&+\left(\mathcal{A}{\bf E}_{xz}(-1){\bf K}^T\right)^T = 0 \text{, and} \label{eqn:ExFxF_th1p1} \ \\
		&\mathcal{A}{\bf \Sigma}_{x^Fe}\mathcal{A}^T-{\bf \Sigma}_{x^Fe}+\left({\bf I}_n-{\bf K}{\bf C}\right){\bf B}{\bf \Sigma}_e{\bf B}^T\left({\bf I}_n-{\bf K}{\bf C}\right)^T = 0. \cr
		\label{eqn:ExFxF_th1p2}
	\end{align}
	Here $\mathcal{A} =\left({\bf I}_n-{\bf K}{\bf C}\right)\left({\bf A}+{\bf B}{\bf L}\right)$, which is assumed to be strictly stable. ${\bf I}_n$ is a identity matrix of size $n \times n$.
\end{lemma}
\begin{proof}[Proof]
	The proof of Lemma~\ref{lemma:sigma_gammas} is provided in Appendix~\ref{apdx:sigma_sq_gamma_attack}.
\end{proof}
\begin{remark}
	Lemma~\ref{lemma:sigma_gammas} provides an analytical formula to derive the value of the non-dependent variance ${\bf \Sigma}_{\widetilde \gamma}$ of the innovation signal ${\widetilde \gamma}$ under an attack. ${\bf \Sigma}_{\widetilde \gamma}$ is used for the sub-optimal CUSUM test, and derivation of the SADD under both the tests.
\end{remark}
\begin{remark}
	Since $\mathcal{A}$ is assumed to be strictly stable, the Lyapunov equations of (\ref{eqn:ExFxF_th1p1}) and (\ref{eqn:ExFxF_th1p2}) will have unique solutions. 
	If $\mathcal{A}$ and ${\bf A}_a$ are not diagonalizable, then ${\bf E}_{xz}\left(-1\right)$ can be evaluated numerically by taking a large number of terms for the summation of (\ref{eqn:Exz1_original}), until the rest of the terms become negligible. 
\end{remark}
\begin{remark}
	The attacker's system parameters ${\bf A}_a$ and ${\bf Q}_a$ can be estimated from the observations. 
\end{remark}
\begin{corollary} \label{cor:Exzm1}
	With the assumption that $\mathcal{A}$ and ${\bf A}_a$ are diagonalizable, ${\bf E}_{xz}(-1)$ will take the following form
	\begin{align}
		&{\bf E}_{xz}(-1)={\bf U}_{\mathcal{A}}{\bf T}_a{\bf U}_a^{-1}{\bf A}_a{\bf E}_{zz}(0)  \label{eqn:Exz1} .
	\end{align}
	Here ${\bf U}_{\mathcal{A}}$ is the eigenvector matrix of $\mathcal{A}$, see (\ref{eqn:svdscriptA}). ${\bf \Sigma}_{\mathcal{A}} = diag\left[ \lambda_{\mathcal{A},1}\ \lambda_{\mathcal{A},2}\ \cdots \right]$ is the eigenvalue matrix of ${\mathcal{A}}$ with the eigenvalues on its main diagonal. ${\bf U}_{a}$ is the eigenvector matrix of ${\bf A}_a$, see (\ref{eqn:svdAa}). ${\bf \Sigma}_{a} = diag\left[ \lambda_{a,1}\ \lambda_{a,2}\ \cdots \right]$ is the eigenvalue matrix of ${\bf A}_a$ with the eigenvalues on its main diagonal. 
	\begin{align}
		\mathcal{A}&={\bf U}_{\mathcal{A}}{\bf \Sigma}_{\mathcal{A}}{\bf U}_{\mathcal{A}}^{-1}. \label{eqn:svdscriptA} \ \\
		{\bf A}_a&={\bf U}_a{\bf \Sigma}_a{\bf U}_a^{-1}. \label{eqn:svdAa}
	\end{align}
	
	The $ij$-th element of the ${\bf T}_a$ matrix is as follows
	\begin{align}
		\left[{\bf T}_a\right]_{ij}&=\frac{\left[{\bf T}\right]_{ij}}{1-\lambda_{\mathcal{A},i}\lambda_{a,j}},  \label{eqn:Ta_mimo}\ \\
		\text{and } {\bf T}&={\bf U}_{\mathcal{A}}^{-1}{\bf K}{\bf U}_a \label{eqn:T_mimo}.
	\end{align}
\end{corollary}
\begin{proof}[Proof]
	Proof of Corollary~\ref{cor:Exzm1} is provided in the Appendix~\ref{apdx:kld_corr}.
\end{proof}
\begin{remark}
	Corollary \ref{cor:Exzm1} provides a way to derive the value of ${\bf E}_{xz}(-1)$ analytically, provided $\mathcal{A}$ and ${\bf A}_a$ are diagonalizable. ${\bf E}_{xz}(-1)$ is used to evaluate ${\bf \Sigma}_{\widetilde \gamma}$. 
\end{remark}
\begin{myth} \label{th:th_kld}
	
	The expected KLD under the optimal CUSUM test $\left(E\left[ D\left(f_{{\widetilde {\bf \gamma}_k}},f_{{ {\bf \gamma}_k}} |\left\{\bar \gamma\right\}_1^{k-1},\left\{{\bf e}\right\}_1^{k-1}\right)\right]\right)$, and the KLD under the sub-optimal CUSUM test $\left(D\left(f_{{\widetilde {\bf \gamma}_k},{\bf e}_{k-1}},f_{{ {\bf \gamma}_k},{\bf e}_{k-1}}\right)\right)$ will be as follows,
	
	\begin{align}
		&E\left[ D\left(f_{{\widetilde {\bf \gamma}_k}},f_{{ {\bf \gamma}_k}} |\left\{\bar \gamma\right\}_1^{k-1},\left\{{\bf e}\right\}_1^{k-1}\right)\right]  \cr 
		&=\frac{1}{2}\left\{tr\left({\bf \Sigma}_\gamma^{-1} {\bf \Sigma}_{\widetilde \gamma}\right) -m - \log\frac{\mid{\bf Q}_a\mid}{\mid{\bf \Sigma}_{\gamma}\mid}\right\} \text{, and}
		\label{eqn:opt_kld} \ \\
		&D\left(f_{{\widetilde {\bf \gamma}_k},{\bf e}_{k-1}},f_{{ {\bf \gamma}_k},{\bf e}_{k-1}}\right)  \cr
		&=\frac{1}{2}\left\{tr\left({\bf \Sigma}_\gamma^{-1} {\bf \Sigma}_{\widetilde \gamma}\right) -m - \log\frac{\mid{{\bf \Sigma}_{\widetilde \gamma}}-{\bf C}{\bf B}{\bf \Sigma}_e{\bf B}^T{\bf C}^T\mid}{\mid{\bf \Sigma}_{\gamma}\mid}\right\}. \label{eqn:subopt_kld} \cr
	\end{align}
	
\end{myth}

\begin{proof}[Proof]
	The proof of Theorem~\ref{th:th_kld} is provided in Appendix~\ref{apdx:kld_mimo}.
\end{proof}
\begin{corollary}
	The difference between the expected KLD and the KLD is $\log\frac{\mid{{\bf \Sigma}_{\widetilde \gamma}}-{\bf C}{\bf B}{\bf \Sigma}_e{\bf B}^T{\bf C}^T\mid}{\mid{\bf Q}_{a}\mid}$, which corresponds to the optimality gap between the optimal and sub-optimal CUSUM tests. From (\ref{eqn:gamma_attack_apndx}), exploiting suitable independence properties of the involved random processes, it can be shown that ${\bf \Sigma}_{\widetilde \gamma}-{\bf C}{\bf B}{\bf \Sigma}_e{\bf B}^T{\bf C}^T \ge  {\bf Q}_{a}$. By  eigenvalue comparison of the positive semidefinite matrices 
	${\bf \Sigma}_{\widetilde \gamma}-{\bf C}{\bf B}{\bf \Sigma}_e{\bf B}^T{\bf C}^T$ and ${\bf Q}_{a}$,	 we can say ${\mid{{\bf \Sigma}_{\widetilde \gamma}}-{\bf C}{\bf B}{\bf \Sigma}_e{\bf B}^T{\bf C}^T\mid} \ge {\mid{\bf Q}_{a}\mid}$, which ensures the optimality gap is  positive.  
\end{corollary}
\begin{proof}[Proof]
	The proof simply follows by subtracting (\ref{eqn:subopt_kld}) from (\ref{eqn:opt_kld}). 
\end{proof}
\begin{remark}
	The expected KLD and the KLD under the optimal and sub-optimal test, respectively, are mostly dependent on the non-dependent variances of the innovation signals ${\bf \Sigma}_{\gamma}$ and ${\bf \Sigma}_{\widetilde \gamma}$ before and after an attack. They also depend on a few system and noise parameters. 
\end{remark}

\begin{remark}
	Instead of taking the joint distribution of the innovation signal and the watermarking signal, if the optimal CUSUM test is performed using the dependent distribution of the innovation signal only, then the expected KLD will take the form of (\ref{eqn:th_kld_only_inno}). While a detailed proof cannot be accommodated due to space constraints, here we use simple intuitive arguments to explain why the expected KLD of (\ref{eqn:th_kld_only_inno}) reduces compared to the optimal KLD using the joint conditional distribution of the innovation signal and the watermarking signal  (\ref{eqn:opt_kld}). An investigation of the KLD expression reveals that the numerator can be described as negative conditional differential entropy, which increases with further conditioning with respect to the watermarking signal, and the denominator (due to the Gaussian property of the distribution of the innovations) can be described as the conditional variance which decreases with further conditioning, thus increasing the KLD overall.  
	The increase in KLD results in quicker attack detection on average due to (\ref{eqn:add_lim}).  Equation (\ref{eqn:th_kld_only_inno}) can be derived following the similar steps given in the Appendix~\ref{apdx:opt_cusum} and Appendix~\ref{apdx:kld_mimo}. However, the detailed proof has been omitted due to the space constraints. 
\end{remark}
\begin{table*}\centering
	\begin{tabular}{  m{15cm} }\hline
		{\begin{flalign}
				&E\left[ D\left(f_{{\widetilde {\bf \gamma}_k}},f_{{ {\bf \gamma}_k}} |\left\{\bar \gamma\right\}_1^{k-1}\right)\right] =\frac{1}{2}\left\{tr\left({\bf \Sigma}_\gamma^{-1} \left({\bf \Sigma}_{\widetilde \gamma}- {\bf E}_\mu-{\bf E}_\mu^T\right)\right) -m - \log\frac{\mid{\bf \Sigma}_{{\widetilde {\bf \gamma}_k}|\left\{\bar \gamma\right\}_1^{k-1}}\mid}{\mid{\bf \Sigma}_{\gamma}\mid}\right\}, \label{eqn:th_kld_only_inno} \\
				& \text{where }{\bf \Sigma}_{{\widetilde {\bf \gamma}_k}|\left\{\bar \gamma\right\}_1^{k-1}} = {\bf Q}_a + \left({\bf A}_a{\bf C}-{\bf C}\left({\bf A} + {\bf B}{\bf L} \right) \right){\bf G}\left({\bf A}_a{\bf C} -{\bf C}\left({\bf A} + {\bf B}{\bf L} \right) \right)^T+{\bf C}{\bf B}{\bf \Sigma}_e{\bf B^T}{\bf C}^T \text{,} \nonumber \\
				&{\bf G} = \sum_{i=2}^{k-1}\left({\bf A} + {\bf B}{\bf L} \right)^{i-1}{\bf B}{\bf \Sigma}_e{\bf B^T}\left[\left({\bf A} + {\bf B}{\bf L} \right)^{i-1} \right]^T, \text{and} \\
				&{\bf E}_\mu=\left({\bf A}_a{\bf C}-{\bf C}\left({\bf A} + {\bf B}{\bf L} \right) \right)\sum_{j=1}^{k-1}\sum_{i=2}^{j+1}\left({\bf A} + {\bf B}{\bf L} \right)^{i-1}{\bf K}{\bf E}_{\gamma e}\left(j-i+1 \right){\bf B}^T\left[\left({\bf A} + {\bf B}{\bf L} \right)^{j-1}\right]^T\left[\left({\bf A}_a{\bf C}-{\bf C}\left({\bf A} \right. \right. \right. \nonumber \\
				& \left. \left. \left. + {\bf B}{\bf L} \right) \right)\right]^T + \left({\bf A}_a-{\bf C}\left({\bf A} + {\bf B}{\bf L} \right){\bf K} \right)\sum_{j=1}^{k-1}{\bf E}_{\gamma e}\left(j \right) {\bf B}^T\left[\left({\bf A} + {\bf B}{\bf L} \right)^{j-1}\right]^T\left[\left({\bf A}_a{\bf C}-{\bf C}\left({\bf A} + {\bf B}{\bf L} \right) \right)\right]^T, \\
				&{\bf E}_{\gamma e}\left(j \right)= \begin{cases}-{\bf C}\left({\bf A} + {\bf B}{\bf L} \right){\cal A}^{j-2}\left({\bf I}_n - {\bf K}{\bf C}\right){\bf B}{\bf \Sigma}_e & \text{if }j >1 \\ {\bf 0} & \text{otherwise.}\end{cases}
		\end{flalign}} \\
		\hline \end{tabular}\label{NS_eqt}\end{table*}

\begin{myth}\label{th:deltaLQG}
	The increase in the LQG cost ($\Delta LQG$) over the optimal LQG cost, when there is no attack, due to the addition of the watermarking signal is related to the watermarking signal covariance matrix  ${\bf \Sigma}_e$ as follows, 
	\begin{align}
		&\Delta LQG=tr\left({\bf H}{\bf \Sigma}_e\right) \label{eqn:deltaLQG} \ \\
		&\text{where }{\bf H} = {\bf B}^T{\bf \Sigma}_L{\bf B}+{\bf U}  \label{eqn:H}
	\end{align}
	and ${\bf \Sigma}_L$ is the solution to the Lyapunov equation
	\begin{align}
		\left({\bf A}+{\bf B}{\bf L}\right)^T{\bf \Sigma}_L\left({\bf A}+{\bf B}{\bf L}\right)-{\bf \Sigma}_L+{\bf L}^T{\bf U}{\bf L} +{\bf W}=0.
		\label{eqn:Sigma_L}
	\end{align}
\end{myth}
\begin{proof}[Proof]
	The theorem can be proved easily using the Theorem 2 from \cite{Mo2015}, considering the iid watermarking as a special case of the hidden Markov model (HMM). 
\end{proof}
\begin{remark}
	Since the closed loop system $\left({\bf A}+{\bf B}{\bf L}\right)$ is stable, the Lyapunov equation of (\ref{eqn:Sigma_L}) will have a unique solution. 
\end{remark}
\begin{remark}
	Theorem~\ref{th:deltaLQG} indicates the increase in the LQG control cost due to the addition of the watermarking, \ie, $\Delta LQG$ is a linear function of the elements of the covariance matrix ${\bf \Sigma}_e$ of the added watermarking. 
	The matrix ${\bf H}$ in (\ref{eqn:deltaLQG}) is dependent on the plant and controller parameters. Since the plant and the controller are assumed to be time-invariant, ${\bf H}$ will be a constant matrix during the steady-state operation of the system. Therefore, the increase in the LQG control cost is linear with respect to the covariance matrix, ${\bf \Sigma}_e$, of the watermarking signal. 
\end{remark}

\subsection{Multiple Input Single Output Systems}
\label{subsec:miso_system}
In this subsection, a simplified case of the MIMO system, \ie, the MISO system is studied to get better structural understanding and insights. Lemma~\ref{le:kld_miso} provides the expressions for the expected KLD and KLD under the optimal and sub-optimal CUSUM tests, respectively, which are the simplified version of the KLD expressions provided in Theorem~\ref{th:th_kld}. The following attack model is assumed for the MISO system, which is a special case of the stochastic linear attack model given in (\ref{eqn:hidden_states_main}), 
\begin{equation}
	\begin{aligned}
		&E\left[z_k^2 \right] = \sigma_z^2 \text{, and}  \cr
		&E\left[z_kz_{k-k_0} \right] = \rho^{k_0}\sigma_z^2 \text{, }\rho < 1.
		\label{eqn:attack_model_miso}
	\end{aligned}
\end{equation}
Therefore, ${\bf A}_a = \rho \text{, and } {\bf Q}_a = \left(1 - \rho^2 \right)\sigma_z^2$.
\begin{lemma}\label{le:kld_miso}
	For a MISO system, the expected KLD $E\left[ D\left(f_{{\widetilde {\bf \gamma}_k}},f_{{ {\bf \gamma}_k}} |\left\{\bar \gamma\right\}_1^{k-1},\left\{{\bf e}\right\}_1^{k-1}\right)\right]$ under the optimal CUSUM test, and the KLD $D\left(f_{{\widetilde {\bf \gamma}_k},{\bf e}_{k-1}},f_{{ {\bf \gamma}_k},{\bf e}_{k-1}}\right)$ under the sub-optimal CUSUM test will be as follows, 
	\begin{align}
		&E\left[ D\left(f_{{\widetilde {\bf \gamma}_k}},f_{{ {\bf \gamma}_k}} |\left\{\bar \gamma\right\}_1^{k-1},\left\{{\bf e}\right\}_1^{k-1}\right)\right] \cr 
		&=\frac{1}{2}\left\{\frac{\sigma_{\widetilde \gamma}^2}{\sigma_{ \gamma}^2} -1 - \log\frac{(1-\rho^2)\sigma_z^2}{{\bf \sigma}_{\gamma}^2}\right\} \label{eqn:opt_kld_miso}, \text{ and} \ \\
		&{D\left(f_{{\widetilde {\bf \gamma}_k},{\bf e}_{k-1}},f_{{ {\bf \gamma}_k},{\bf e}_{k-1}} \right)}= \cr
		& \frac{1}{2}\left\{\frac{\sigma_{\widetilde \gamma}^2}{\sigma_{ \gamma}^2}-1- \log\frac{\sigma_{\widetilde \gamma}^2-{\bf C}{\bf B}{\bf \Sigma}_e{\bf B}^T{\bf C}^T}{\sigma_{ \gamma}^2} \right\}
		\label{eqn:subopt_kld_miso}
	\end{align}
	where the attack model is given by (\ref{eqn:attack_model_miso}). $\sigma_{ \gamma}^2$ and $\sigma_{\widetilde \gamma}^2$ are the scalar variances of the innovation signals $\gamma_k$ and $\widetilde \gamma_k$ before and after the attack, respectively. Hence,
	\begin{align}
		\sigma_{\gamma}^2&={\bf C}{\bf P}{\bf C}^T+{R} \text{, and} \ \\
		\sigma_{\widetilde \gamma}^2 &= M_z\sigma_z^2+\text{tr}\left({\bf M}_e{\bf \Sigma}_e\right) \label{eqn:sigma_sq_attack_th}
	\end{align}
	where $R$ and $M_z$ are scalar quantities. $M_z$ and ${\bf M}_e$ will take the following forms,
	\begin{align}
		&M_z=1-2{\bf C}\left({\bf A}+{\bf B}{\bf L}\right)\left({\bf I}_n-\rho\mathcal{A}\right)^{-1}{\bf K} \rho+ \nonumber \\
		&{\bf C}\left({\bf A}+{\bf B}{\bf L}\right){\bf \Sigma}^z_{x^F}\left({\bf A}+{\bf B}{\bf L}\right)^T{\bf C}^T \text{, and} \label{eqn:M1} \ \\
		&{\bf M}_e = {\bf B^T}\left({\bf I}_n - {\bf K}{\bf C}\right)^T{\bf \Sigma}^e_{x^F}\left({\bf I}_n - {\bf K}{\bf C}\right){\bf B}+{\bf B}^T{\bf C}^T{\bf C}{\bf B} \label{eqn:M2} \nonumber \\
	\end{align}
	where ${\bf \Sigma}^z_{x^F}$ and ${\bf \Sigma}^e_{x^F}$ are the solutions to the Lyapunov equations, 
	\begin{align}
		&{\bf \mathcal A}{\bf \Sigma}^z_{x^F}{\bf \mathcal A}^T-{\bf \Sigma}^z_{x^F}+ {\bf K}{\bf K}^T+{\bf \mathcal{A}}\left[{\bf I}_n-\rho{\bf \mathcal A}\right]^{-1}{\bf K}{\bf K}^T\rho \cr
		&+\left[{\bf \mathcal{A}}\left[{\bf I}_n-\rho{\bf \mathcal A}\right]^{-1}{\bf K}{\bf K}^T\rho\right]^T=0, \label{eqn:Sigma_xz} \ \\
		& \text{and} \cr
		&\mathcal{A}^{T}{\bf \Sigma}^e_{x^F}\mathcal{A}-{\bf \Sigma}^e_{x^F}+\left({\bf A}+{\bf B}{\bf L}\right)^T{\bf C}^T{\bf C}\left({\bf A}+{\bf B}{\bf L}\right)=0 \cr
		& \text{respectively.}	
		\label{eqn:Sigma_xe} 
	\end{align}
	Furthermore, $\Delta LQG$ coincides with Theorem~\ref{th:deltaLQG}.
\end{lemma}
\begin{proof}
	(\ref{eqn:opt_kld_miso}) and (\ref{eqn:subopt_kld_miso}) can be derived directly by replacing the variables from (\ref{eqn:opt_kld}) and (\ref{eqn:subopt_kld}) by their MISO system counterparts. Therefore, only the derivation of $\sigma_{\widetilde \gamma}^2$ is provided in Appendix~\ref{apdx:kld_miso}. 
\end{proof}

\begin{remark}
	The expected KLD (\ref{eqn:opt_kld_miso}) and the KLD (\ref{eqn:subopt_kld_miso}) are convex functions in ${\bf \sigma}_z^2$. The convexity can be proved by taking the first and second derivative of (\ref{eqn:opt_kld_miso}) and (\ref{eqn:subopt_kld_miso})  with respect to ${\bf \sigma}_z^2$. The minimum value of the KLD will be achieved for $\sigma_z^{*2}=$ $\frac{\sigma_{\gamma}^2}{M_z}$ and $\frac{\sigma_{\gamma}^2-tr\left(\left({\bf M}_e-{\bf B}^T{\bf C}^T{\bf C}{\bf B}\right){\bf \Sigma}_e\right)}{M_z}$ for the optimal and sub-optimal tests, respectively. Therefore, we can conclude the KLD is not always increasing with the attacker signal power $\sigma_z^2$; it depends also on the power of the watermarking signal for the sub-optimal test. However,  ${\bf \sigma}_z^{*2}$ for the optimal test does not depend on the watermarking signal power. In fact, the attacker can modify ${\bf \sigma}_z^2$ to ${\bf \sigma}_z^{*2}$ to reduce the KLD which in turn reduces the probability of detection. On the other hand, the control system designer can choose $tr\left(\left({\bf M}_e-{\bf B}^T{\bf C}^T{\bf C}{\bf B}\right){\bf \Sigma}_e\right) \ge \sigma_{\gamma}^2$ for the sub-optimal case, so that the KLD will always increase with the attacker signal power. However, under the optimal test, the control system designer can not do much to avoid this situation. On the other hand, for the sub-optimal test, the attacker needs to know $\Sigma_e$ to derive $\sigma_z^{*2}$. 
\end{remark}

\subsection{Optimum Watermarking Signal for MISO systems}
\label{subsec:opt_e}
By increasing the watermarking power ${\bf \Sigma}_e$, we can improve the KLD, but at the same time, it also increases the control cost, \ie, $\Delta LQG$ becomes larger. Therefore, we want to find the optimal ${\bf \Sigma}_e$, say ${\bf \Sigma}_e^*$, which will maximize the KLD subject to an upper bound on $\Delta LQG$. The optimization problem is formulated as follows,
\begin{align}
	\max_{{\bf \Sigma}_e} &\  E\left[ D\left(f_{{\widetilde {\bf \gamma}_k}},f_{{ {\bf \gamma}_k}} |\left\{\bar \gamma\right\}_1^{k-1},\left\{{\bf e}\right\}_1^{k-1}\right)\right] \text{or}  \cr
	\max_{{\bf \Sigma}_e} &\  D\left(f_{\widetilde \gamma_k,{\bf e}_{k-1}},f_{ \gamma_k,{\bf e}_{k-1}}\right) \\
	\textrm{s.t.}\  & \Delta LQG \le J \\
	& {\bf \Sigma}_e \ge 0
\end{align}
where $J$ is a user choice. 
The positive semi-definite ${\bf \Sigma}_e$ matrix can be decomposed by the eigenvalue decomposition as 
\begin{align}
	{\bf \Sigma}_e={\bf V}_e{\bf \Lambda}_e{\bf V}_e^T,
\end{align}
where ${\bf V}_e$  is the orthonormal eigenvector matrix and ${\bf \Lambda}_e$ is the diagonal eigenvalue matrix. If we assume that ${\bf V}_e$ is known apriori, then we only need to find the optimum ${\bf \Lambda}_e$ which is a diagonal matrix. 
\begin{myth}\label{th:opt_diagonal_Sigma_e}
	The optimum diagonal ${\bf \Lambda}_e$ that will maximize the expected KLD under the optimal CUSUM test or the KLD under the sub-optimal CUSUM test subject to $\Delta LQG \le J$ will have only one non-zero element on its main diagonal.
\end{myth}
\begin{proof}
	The proof of Theorem~\ref{th:opt_diagonal_Sigma_e} is provided in Appendix~\ref{apdx:optimization}.
\end{proof}
In the light of Theorem~\ref{th:opt_diagonal_Sigma_e}, we search for the optimum ${\bf \Sigma}_e$ in the class of rank one positive semi-definitive matrices of the following form
\begin{align}
	{\bf \Sigma}_e=\lambda_e {\bf v}_e {\bf v}_e^T,
	\label{eqn:sigma_e_rank_1}
\end{align}
where $\lambda_e$ is the non-zero eigenvalue and ${\bf v}_e$ is the corresponding eigenvector. We modify (\ref{eqn:sigma_e_rank_1}) to represent it in the following form
\begin{align}
	{\bf \Sigma}_e= {\bf v}_{\lambda} {\bf v}_{{\lambda}}^T \text{, where } {\bf v}_{\lambda} = {\sqrt \lambda_e}{\bf v}_e.
	\label{eqn:sigma_e_rank_1_2nd}
\end{align}
Finally, the optimization problem becomes, 
\begin{align}
	\max_{{\bf v}_{\lambda}} &\  E\left[ D\left(f_{{\widetilde {\bf \gamma}_k}},f_{{ {\bf \gamma}_k}} |\left\{\bar \gamma\right\}_1^{k-1},\left\{{\bf e}\right\}_1^{k-1}\right)\right] \text{or}  \cr
	\max_{{\bf v}_{\lambda}} &\  D\left(f_{\widetilde \gamma_k,{\bf e}_{k-1}},f_{ \gamma_k,{\bf e}_{k-1}}\right) \label{eqn:cost_fun_v}\\
	\textrm{s.t.}\  & \Delta LQG \le J . \label{eqn:const_fun_v}
\end{align}
The optimization problem can be solved using different methods such as the sequential quadratic programming (SQP) \cite{Boggs1995}, the interior point method \cite{Forsgren2002}, etc. We have also provided a simple gradient descent based algorithm to solve the optimization problem (\ref{eqn:cost_fun_v})-(\ref{eqn:const_fun_v}) in Appendix-\ref{apdx:opt_algo}.

The cost function under the optimal CUSUM test can be simplified. Maximization of $E\left[ D\left(f_{{\widetilde {\bf \gamma}_k}},f_{{ {\bf \gamma}_k}} |\left\{\bar \gamma\right\}_1^{k-1},\left\{{\bf e}\right\}_1^{k-1}\right)\right]$ with respect to ${\bf v}_{\lambda}$ is the same as maximizing the following function with respect to ${\bf v}_{\lambda}$.
\begin{equation}
	\begin{aligned}
		&{\bf v}_{\lambda}^T{\bf H}_{KLD}{\bf v}_{\lambda}  \\
		&\text{where}  \\
		&{\bf H}_{KLD} = {\bf B}^T \left({\bf I}_n - {\bf K}{\bf C}\right)^T{\cal L}_e\left({\bf I}_n - {\bf K}{\bf C}\right){\bf B}+{\bf B}^T{\bf C}^T{\bf C}{\bf B}
		\label{eqn:H_KLD_apn}
	\end{aligned}
\end{equation}
and ${\cal L}_e$ is the solution to the Lyapunov equation
\begin{align}
	{\cal{A}}^T{\cal L}_e{\cal{A}}-{\cal L}_e+\left({\bf A}+{\bf B}{\bf L}\right)^T{\bf C}^T{\bf C}\left({\bf A}+{\bf B}{\bf L}\right)=0
	\label{eqn:Le}
\end{align}
Since the matrix ${\cal A}$ is assumed to be strictly stable, the Lyapunov equation of (\ref{eqn:Le}) will have unique solution. The derivations are provided in Appendix-\ref{apdx:opt_algo}. (\ref{eqn:H_KLD_apn}) and (\ref{eqn:Le}) can be simplified for the system with relative degree higher than one, since $\bf{CB}=0$.


\section{Numerical Results}
\label{sec:numerical_results}
In this section, we will illustrate and validate different aspects of the theorems and lemmas presented in this paper using two different system models. The two different systems are (i) System-A: A second-order open-loop unstable MISO system, and (ii) System-B: A fourth-order open-loop stable MIMO system. The system parameters are provided in Appendix~\ref{apdx:system_param}. System-B is a linearized minimum phase quadruple tank system which is used previously to test the deception attack detection schemes in the literature \cite{Johansson1998}. Only the level sensor gains are altered to make the magnitude of the product ${\bf C}{\bf B}$ numerically significant.

\subsection{Tradeoff between SADD and $\Delta LQG$ under optimal CUSUM test}
Figure~\ref{fig:add_vs_Dlqg_system_opt_cusun} shows the tradeoff between the SADD and the increase in the LQG control cost $\Delta LQG$ for System-A and System-B under the optimal CUSUM test (\ref{eqn:th_cusum_opt}). We plot the derived SADD using the theory developed in this paper, and the estimated SADD from Monte-Carlo (MC) simulation, where ${\bf \Sigma}_e$ is assumed to be diagonal and all the watermarking signals have equal power. An increase in LQG cost results in quicker detection. 

\begin{figure}[h!]
	\centering
	\includegraphics[width=\figwidth]{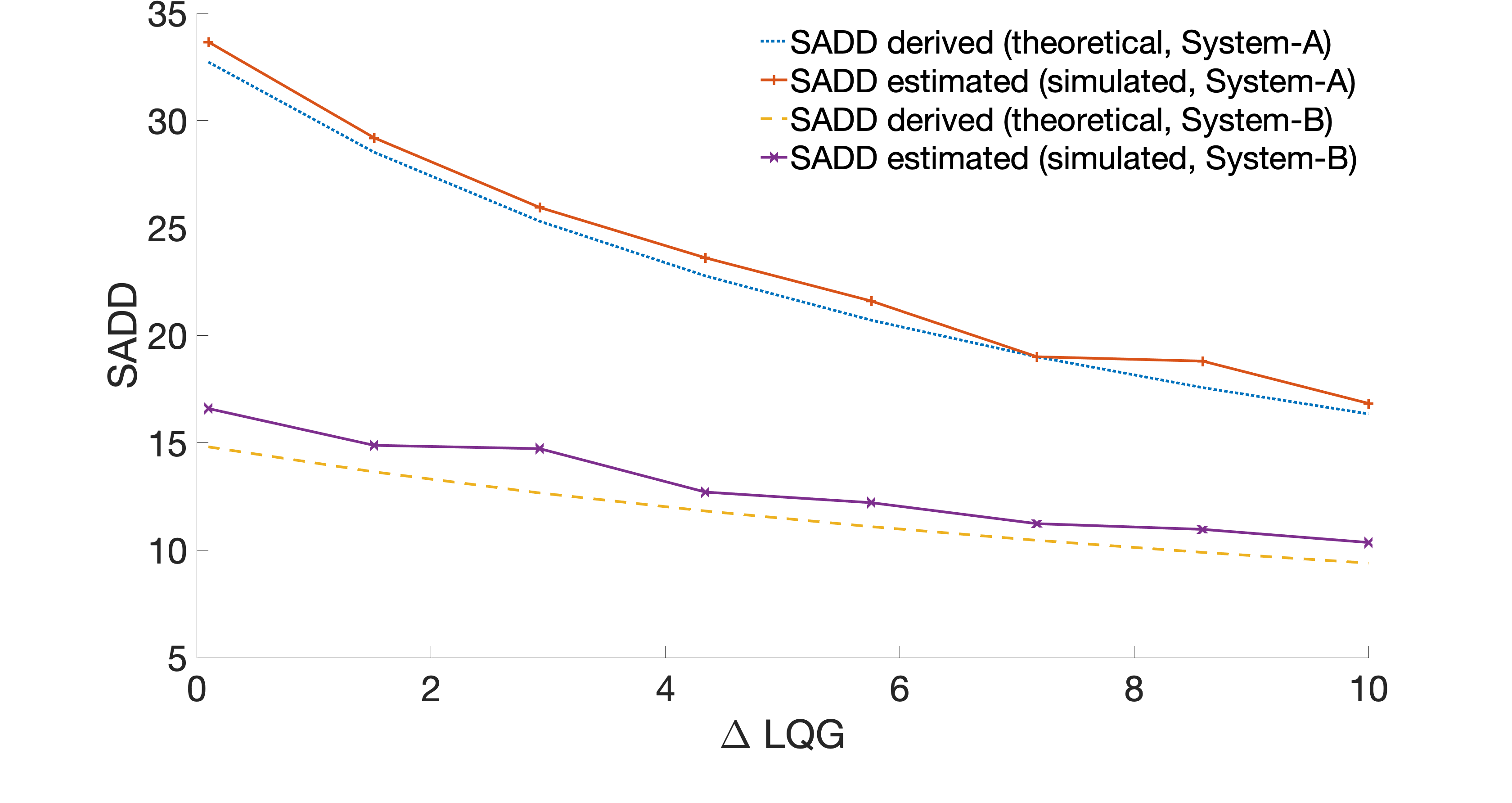}   
	\caption{SADD vs. $\Delta LQG$ plot for System-A and System-B.}
	\label{fig:add_vs_Dlqg_system_opt_cusun}
\end{figure}

\subsection{Benefit of using the joint distribution}
The choice of the joint distribution of the innovation signal and the watermarking signal improves the KLD for a fixed $\Delta LQG$ value compared to the case where the joint distribution is not considered. Therefore, we achieve the same SADD at a lower cost. As shown in Fig.~\ref{fig:compare_sadd_joint_vs_single_dist_opt_cusum}, the same theoretical SADD can be achieved at 64\% (approx.) reduced $\Delta LQG$ for System-A between the $\Delta LQG_1$ and $\Delta LQG_2$ points under the optimal CUSUM test. The percentage reduction in $\Delta LQG$ is evaluated as $\frac{\Delta LQG_2 - \Delta LQG_1}{\Delta LQG_1}\times 100\%$.

\begin{figure}[h!]
	\centering
	\includegraphics[width=\figwidth]{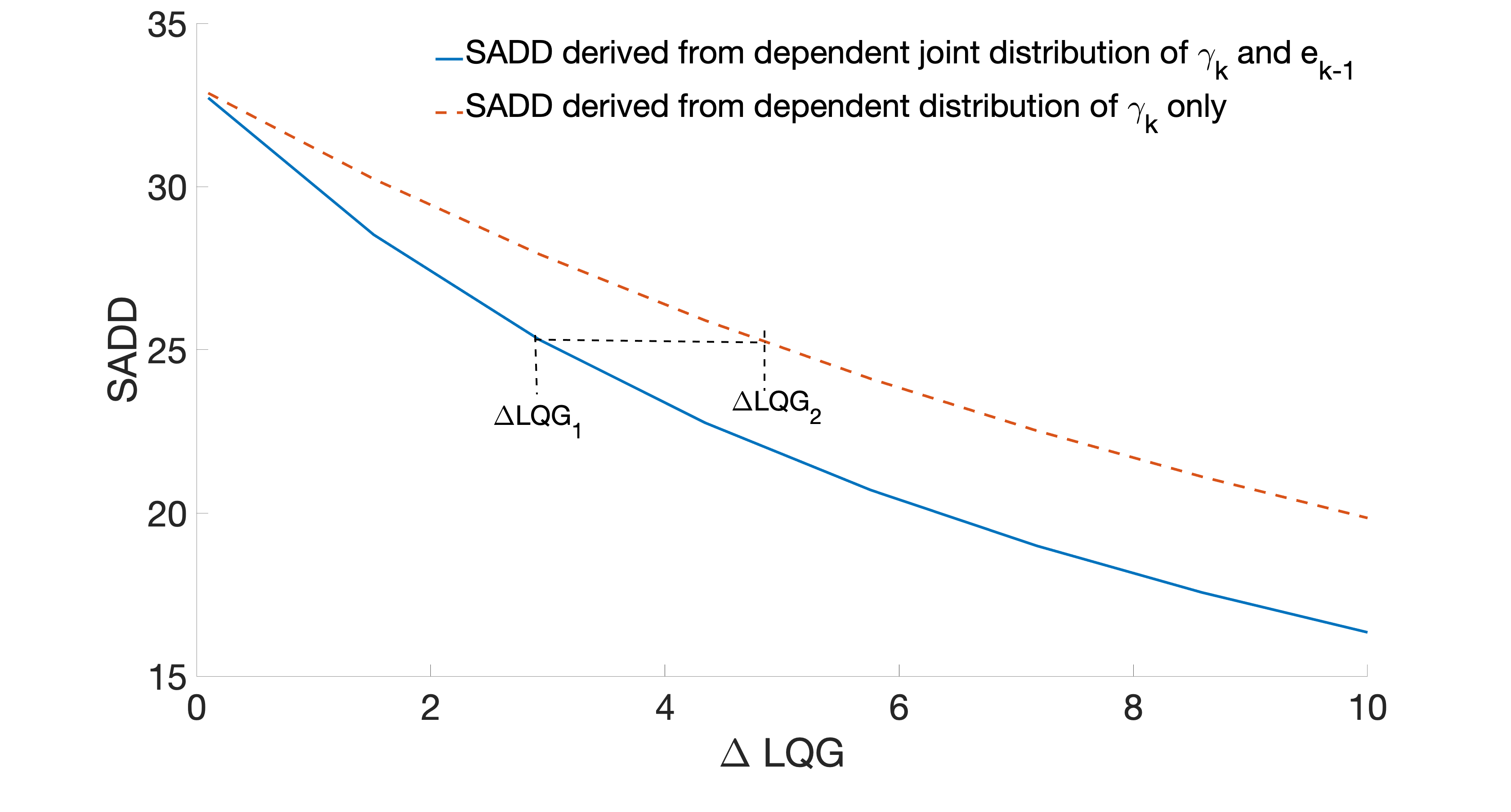}
	\caption{Comparison of SADD vs. $\Delta LQG$ plots between the optimal CUSUM detection schemes using joint and single distributions for System-A.}
	\label{fig:compare_sadd_joint_vs_single_dist_opt_cusum}
\end{figure}
\subsection{Convexity of KLD with respect to $\sigma_z^2$}
Figure~\ref{fig:KLD_vs_sigma_sq_z} shows how the KLD varies with $\sigma_z^2$ for System-A under the optimal and sub-optimal CUSUM tests. The KLD appears to be a convex function with respect to $\sigma_z^2$, and the minimum points are the same as predicted by our theory, see Fig.~\ref{fig:KLD_vs_sigma_sq_z}. We assume, $\Delta LQG = 100$, and ${\bf \Sigma}_e$ to be diagonal and both the watermarking signals to have equal power. Figure~\ref{fig:KLD_vs_sigma_sq_z} can also be interpreted as, for the selected $\Delta LQG$ we can detect an attack equally well for a small ${\bf \sigma}_z^2$ as for a significantly larger ${\bf \sigma}_z^2$.
\begin{figure}[h!]
	\centering
	\includegraphics[width=\figwidth]{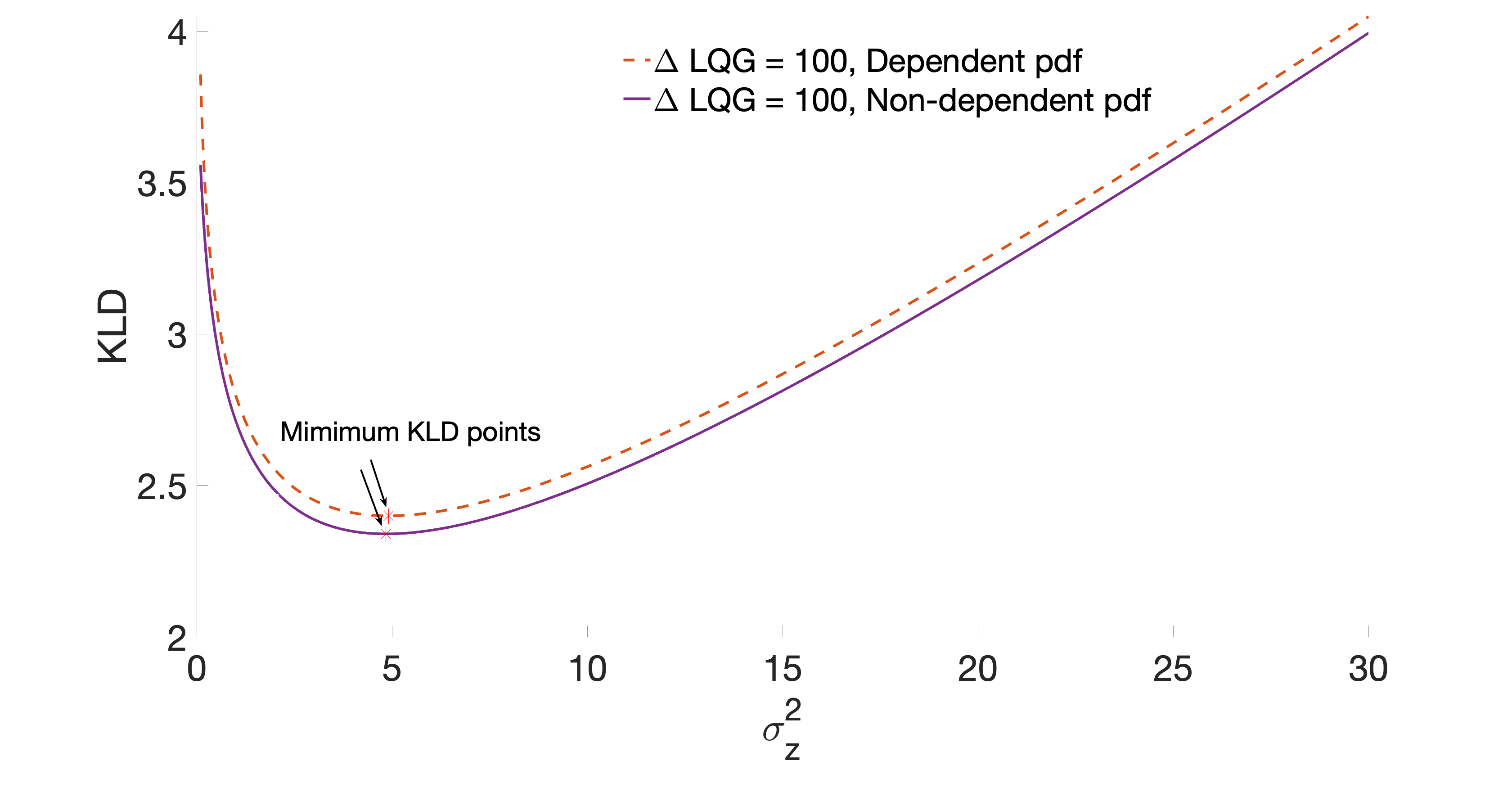}
	\caption{KLD vs. $\sigma_z^2$ plots for System-A.}
	\label{fig:KLD_vs_sigma_sq_z}
\end{figure}

\subsection{Optimum vs non-optimum ${\bf \Sigma}_{e}$}
We optimize the ${\bf \Sigma}_e$ under the optimal test using the method in Subsection~\ref{subsec:opt_e}. Figure~\ref{fig:add_vs_Dlqg_system_opt_e_opt_cusum} shows the SADD vs $\Delta LQG$ plots using the optimized ${\bf\Sigma}_e$ and a diagonal ${\bf \Sigma}_e$ with equal signal power under the optimal CUSUM test. We plot the derived SADD using our theory and the estimated SADD from MC simulation for optimized ${\bf \Sigma}_e$ and non-optimized ${\bf \Sigma}_e$ in the figure. It is evident that optimizing ${\bf \Sigma}_e$ helps in improving SADD for a fixed upper limit on $\Delta LQG$. On the other hand, we can comment that the same theoretical SADD can be achieved at 336\% (approx.) reduced $\Delta LQG$ for System-A between the points $\Delta LQG_1$ and $\Delta LQG_2$.

\begin{figure}[h!]
	\centering
	\includegraphics[width=\figwidth]{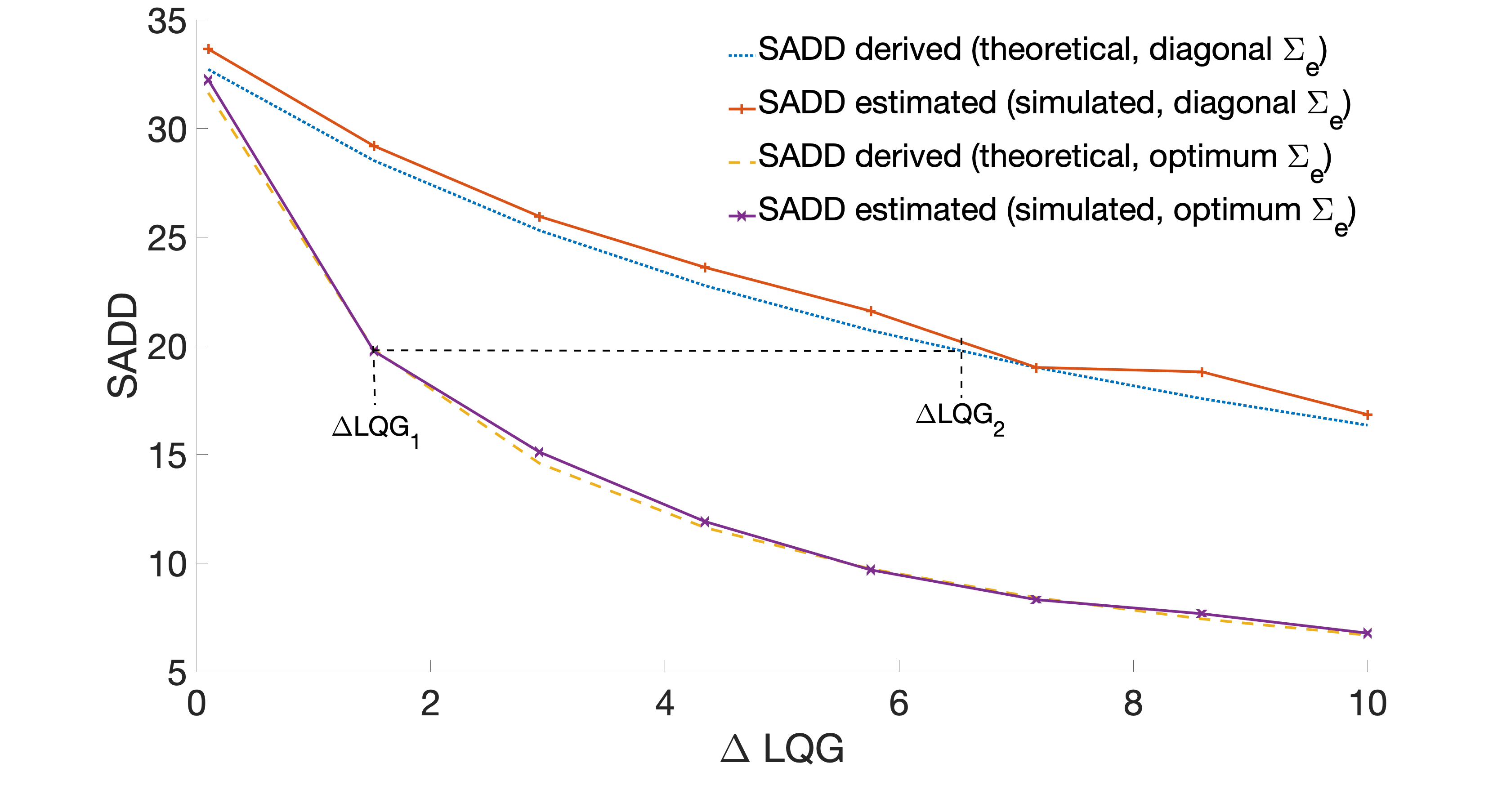}
	\caption{SADD vs. $\Delta LQG$ plot for System-A with optimum and non-optimum ${\bf \Sigma}_e$ under optimal CUSUM test.}
	\label{fig:add_vs_Dlqg_system_opt_e_opt_cusum}
\end{figure}

\subsection{Optimal vs sub-optimal CUSUM}
Figure~\ref{fig:compare_optimal_vs_subopt_cusum} illustrates the advantage of performing the optimal CUSUM test with dependent PDFs over the sub-optimal CUSUM test using the non-dependent PDFs for System-A. For both the plots, optimum ${\bf \Sigma_e}$ has been used for the corresponding cases. Therefore, we can achieve lower SADD for the same $\Delta LQG$ with the optimal CUSUM test compared to the sub-optimal one. The benefit is larger for the lower $\Delta LQG$ values as per the figure. 

\begin{figure}[h!]
	\centering
	\includegraphics[width=\figwidth]{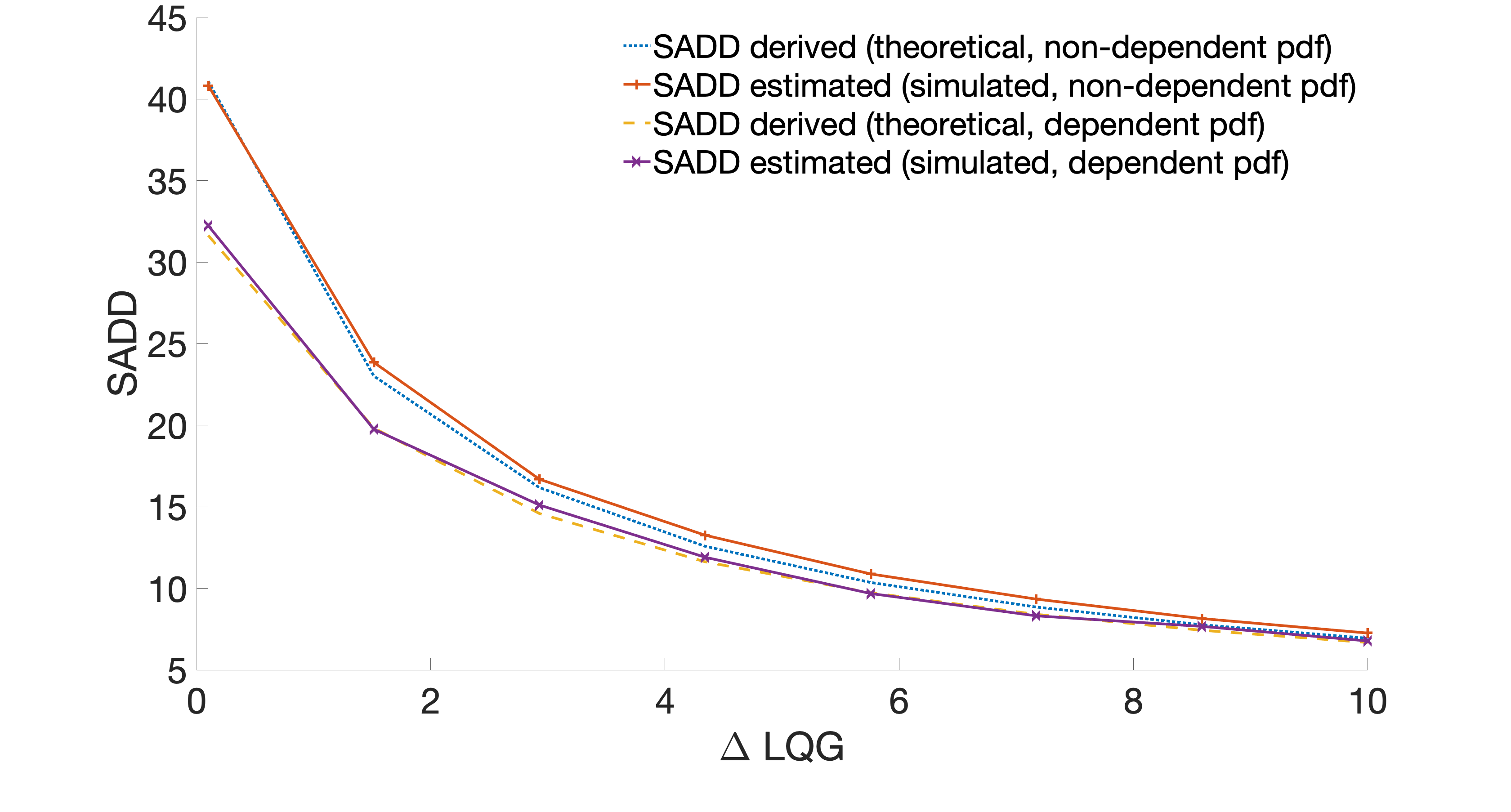}
	\caption{SADD vs. $\Delta LQG$ plot for System-A under optimal and sub-optimal CUSUM tests.}
	\label{fig:compare_optimal_vs_subopt_cusum}
\end{figure}

\subsection{Comparison with optimal NP detector}
We have compared the optimal CUSUM test results with the optimal NP detector based method reported in \cite{Mo2015, Mo2014}. The watermarking signal is taken to be iid, and the ${\bf \Sigma}_e$ is optimized for both the cases. In \cite{Mo2015}, the optimal NP detector rejects the $H_0$ hypothesis in favour of $H_1$ if  
\begin{align}
	&g_{NP,k}\left(\gamma_k, {\bf e}_{k-1},\cdots\right)=\gamma_k^T{\bf \Sigma}_{\gamma}^{-1}\gamma_k \cr
	&-\left(\gamma_k-\mu_{NP,k}\right)^T\left({\bf \Sigma}_{\gamma}+{\bf \Sigma}_{f}\right)^{-1}\left(\gamma_k-\mu_{NP,k}\right) \ge\eta
	\label{eqn:g_np} \\
	&\text{where }\mu_{NP,k} =  -{\bf C}\sum_{i=-\infty}^k{\cal A}^{k-i}{\bf B}{\bf e}_i,  \\
	&{\bf \Sigma}_{f} = {\bf C}{\cal L}_{f}{\bf C}^T  \text{, and} \\
	&{\cal L}_{f} = {\cal A}{\cal L}_{f}{\cal A}^T+{\bf B}{\bf \Sigma}_e{\bf B}^T.
\end{align}
The threshold $\eta$ is estimated by simulation from 
\begin{align}
	P_{\infty}\left\{g_{NP,k}(\cdot)\ge \eta\right\}=\alpha
\end{align}
where $P_{\infty}$ denotes the probability under no attack condition, and $\alpha$ is the threshold on the false alarm rate. The false alarm rate is the reciprocal of the ARL \cite{Murguia2016,Tunga2018}. For the method in \cite{Mo2015}, the ADD is estimated as 
\begin{align}
	ADD_{NP}=E\left[\inf\left\{k:g_{NP,k}(\cdot)\ge \eta\right\}\right].
	\label{eqn:ADD_NP}
\end{align}

Figure~\ref{fig:compare_cusum_np_trials} illustrates how the test statistics $gd_k$ and $g_{NP,k}$ vary with time $k$ under the optimal CUSUM (\ref{eqn:th_cusum_opt}) and NP tests for two random trial runs. The thresholds for the corresponding tests are also shown in the figure. When the test statistics crosses the threshold for the first time that is considered as the attack detection point. System-A is used for generating Fig.~\ref{fig:compare_cusum_np_trials}. 
\begin{figure}[h!]
	\centering
	\includegraphics[width=\figwidth]{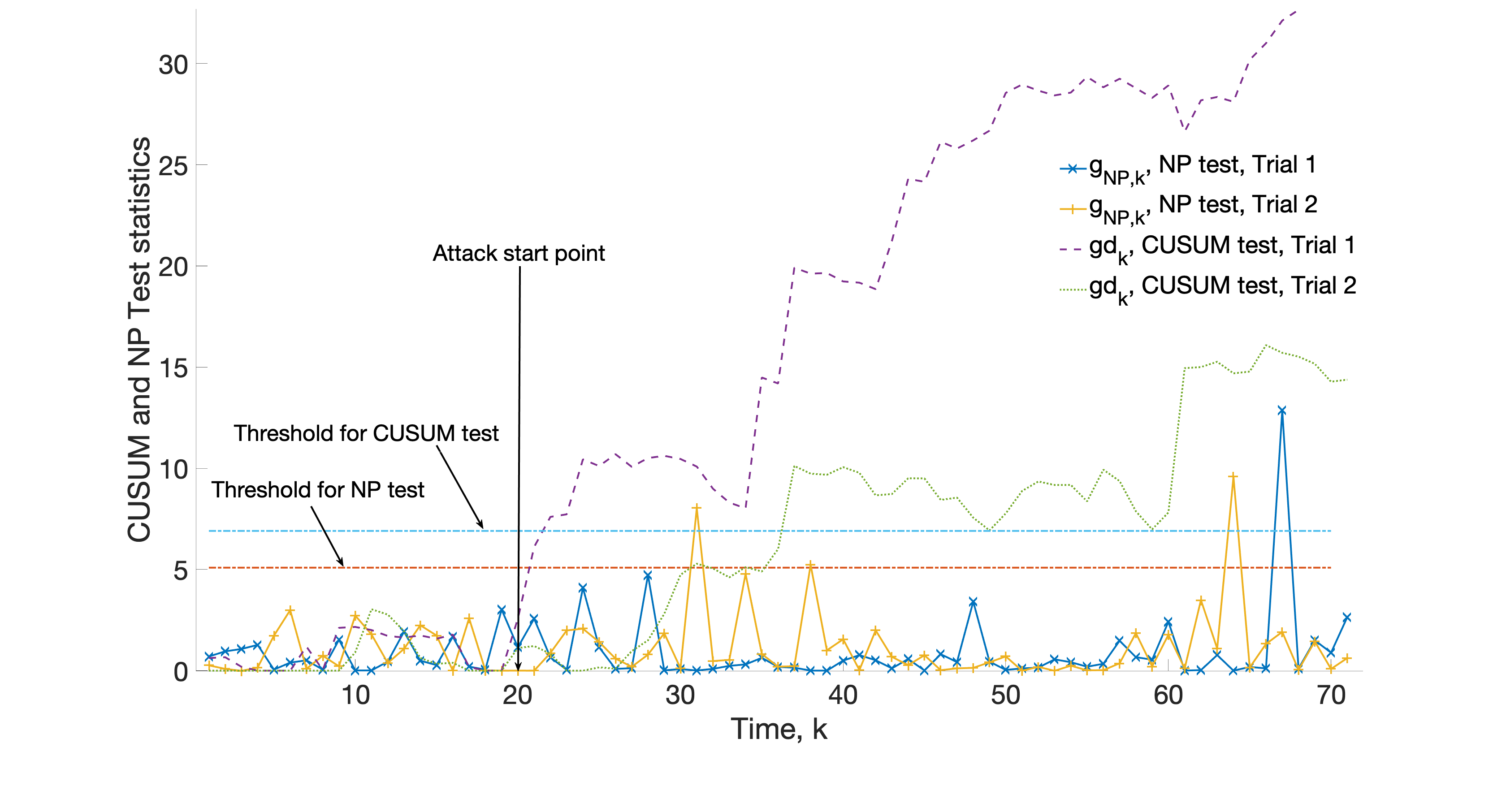}
	\caption{Test statistics under optimal CUSUM test and optimal NP test}
	\label{fig:compare_cusum_np_trials}
\end{figure}
Figure~\ref{fig:compare_cusum_vs_chi_sq} shows the tradeoff between the ADD and the increase in $\Delta LQG$ for System-A under the optimal CUSUM test and the method reported in \cite{Mo2015}. We plot the derived SADD using the theory developed in this paper, the estimated SADD applying the optimal CUSUM test on the simulated data, and the estimated ADD applying the test reported in \cite{Mo2015} on the simulated data. It is clear from the figure that we can achieve lower ADD for the same LQG loss with the method proposed in this paper compared to the one reported in \cite{Mo2015}.
\begin{figure}[h!]
	\centering
	\includegraphics[width=\figwidth]{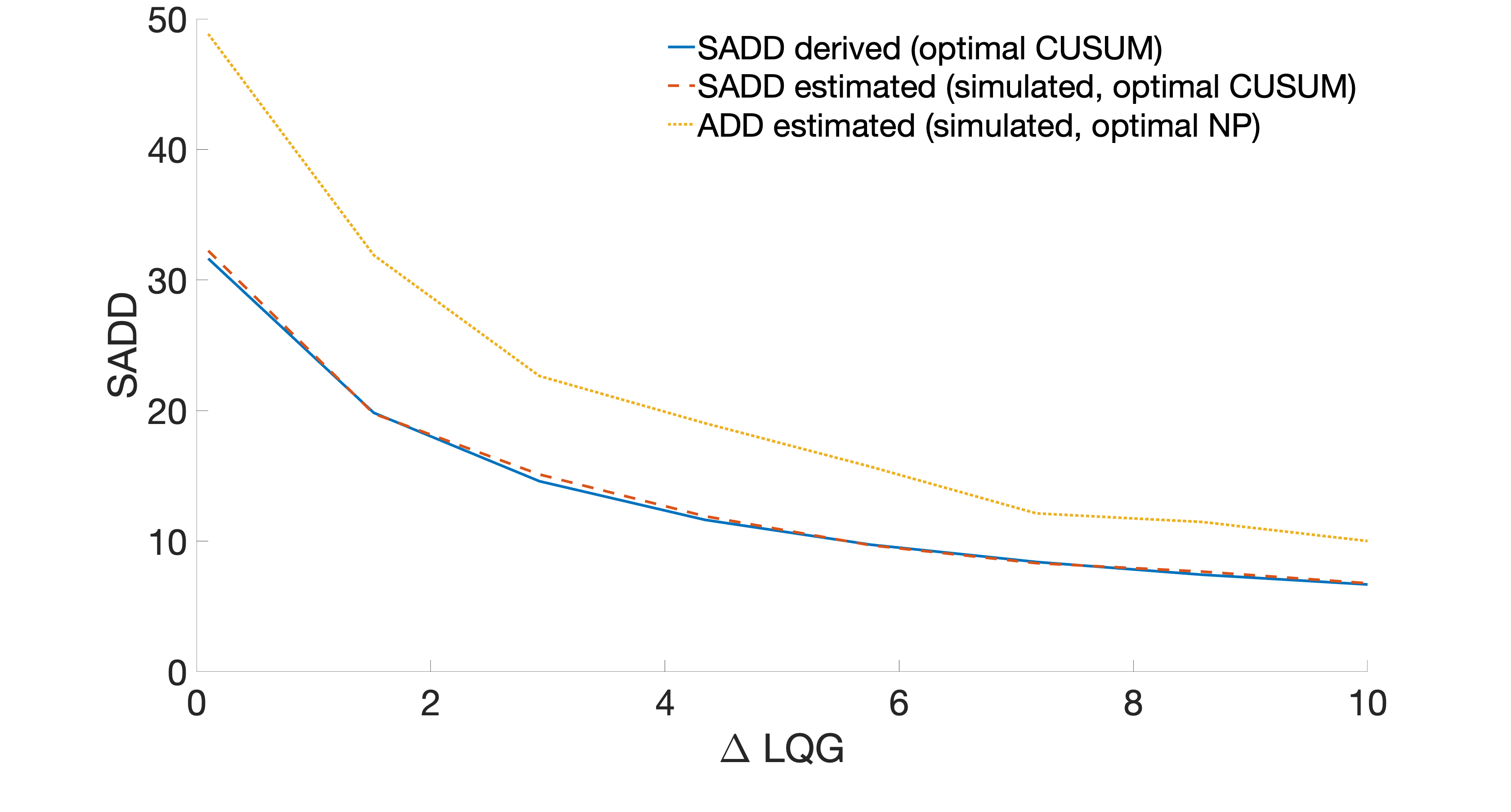}
	\caption{SADD vs. $\Delta LQG$ plot for System-A under optimal CUSUM test and optimal NP test}
	\label{fig:compare_cusum_vs_chi_sq}
\end{figure}


\section{Conclusion}
\label{sec:conclusion}
We have studied the design of the quickest attack detection scheme by adding optimal random watermarking signals, where the attacker replaces the true observations by false data, and tries to cause damage to the NCS. There is a trade-off between the decrease in SADD and the increase in LQG control cost due to the addition of the watermarking signal. We have shown a strategy to find the optimum watermarking signal variance to minimize SADD for a given increase in LQG cost for a MISO system. We found that there is a single optimum eigenvalue and direction for the optimal watermarking signal variance. The relative magnitudes of the attack signal and the watermarking signal also play an important role in attack detection or attack stealthiness. The insights provided in the paper are useful to design a proper watermarking signal. The proposed sequential detection scheme can also be applied for replay attack detection after a few modifications. We have also compared the optimal CUSUM test with the optimal NP test to detect the deception attack and found the optimal CUSUM test to be quicker. In the future, the sequential attack detection scheme can be extended to detect other kinds of attacks as well. The problem of attack detection can also be formulated as a dynamic two-player game between the control system designer and the attacker. This is a topic for future research.

\appendices
\section{Proof of Theorem~\ref{th:opt_cusum} }
\label{apdx:opt_cusum}
Under the optimal CUSUM test, the likelihood ratio from (\ref{eqn:cusum_opt}) can be simplified using the chain rule of probability as
\begin{align}
	&\frac{f_{{\widetilde {\bf \gamma}_k}}\left({{\bar {\bf \gamma}}_k}|\left\{\bar \gamma\right\}_1^{k-1},\left\{{\bf e}\right\}_1^{k-1}\right)f_{{\bf e}_{k-1}}\left({\bf e}_{k-1}\right)}{f_{{ {{\bf \gamma}_k}}}\left({{\bar {\bf \gamma}}_k}\right)f_{{\bf e}_{k-1}}\left({\bf e}_{k-1}\right)} \ \\
	& \text{[} {\bf e}_{k} \text{ is iid and stationary, and } {\bf \gamma}_k \text{ and } {\bf e}_{k-1} \text{ are uncorrelated]} \cr
	& = \frac{f_{{\widetilde {\bf \gamma}_k}}\left({{\bar {\bf \gamma}}_k}|\left\{\bar \gamma\right\}_1^{k-1},\left\{{\bf e}\right\}_1^{k-1}\right)}{f_{{ {{\bf \gamma}_k}}}\left({{\bar {\bf \gamma}}_k}\right)} \text{ [provided } f_{{\bf e}_{k-1}}\left({\bf e}_{k-1}\right) \ne 0 \text{],}
\end{align}
where ${\bar \gamma}_k={ \gamma}_k$ before the attack, and ${\bar \gamma}_k={\widetilde \gamma}_k$ after the attack. The conditional mean (${\bf \mu}_{{\widetilde {\bf \gamma}_k}|\left\{\bar \gamma\right\}_1^{k-1},\left\{{\bf e}\right\}_1^{k-1}}$) and covariance (${\bf \Sigma}_{{\widetilde {\bf \gamma}_k}|\left\{\bar \gamma\right\}_1^{k-1},\left\{{\bf e}\right\}_1^{k-1}}$) of ${\widetilde {\bf \gamma}}_k$ are derived as follows.

The innovation signal under attack from (\ref{eqn:gamma_e_attack}) can be written as (\ref{eqn:gamma_attack_apndx}) after replacing ${\bf z}_k$ by (\ref{eqn:hidden_states_main}),
\begin{equation}
	\widetilde\gamma_k ={\bf w}_{a,k-1}+{\bf A}_a{\bf z}_{k-1}-{\bf C}\left({\bf A}+{\bf B}{\bf L}\right){\bf \hat x}^F_{k-1|k-1}-{\bf C}{\bf B}{\bf e}_{k-1}.
	\label{eqn:gamma_attack_apndx}
\end{equation}
Applying (\ref{eqn:xF_k_k_1}), (\ref{eqn:add_ek}), (\ref{eqn:opt_u}) in (\ref{eqn:kf_state_eqn_attack}) we can write, 
\begin{equation}
	{\hat{\bf{x}}}^F_{k|k}=\left({\bf A}+{\bf B}{\bf L}\right){\hat{\bf{x}}}^F_{k-1|k-1}+{\bf B}{\bf e}_{k-1}+{\bf K}{\widetilde \gamma}_{k-1}. 
	\label{eqn:xf_k_k_apndx_1}
\end{equation}
Using (\ref{eqn:xf_k_k_apndx_1}) recursively we get,
\begin{align}
	&{\hat{\bf{x}}}^F_{k|k}=\left({\bf A}+{\bf B}{\bf L}\right)^{k-1}{\hat{\bf{x}}}_{1|1} \cr 
	&+\sum_{i=1}^{k-1}\left({\bf A}+{\bf B}{\bf L}\right)^{i-1}\left({\bf B}{\bf e}_{k-i-1}+{\bf K}{\bar \gamma}_{k-i} \right) \cr
	& \text{where } {\bar \gamma}_k = { \gamma}_k \text{ for } k<\nu, {\bar \gamma}_k = {\widetilde \gamma}_k \text{ otherwise}.
	\label{eqn:xf_k_k_apndx_2}
\end{align}
Applying (\ref{eqn:gamma_e_attack}) and (\ref{eqn:xf_k_k_apndx_2}) in (\ref{eqn:gamma_attack_apndx}) we get,
\begin{align}
	&\widetilde\gamma_k ={\bf w}_{a,k-1}+\left({\bf A}_a{\bf C}-{\bf C}\left({\bf A}+{\bf B}{\bf L}\right)\right)\left(\left({\bf A}+{\bf B}{\bf L}\right)^{k-2}{\hat{\bf{x}}}_{1|1} \right. \nonumber \\
	& \left. + \sum_{i=1}^{k-2}\left({\bf A}+{\bf B}{\bf L}\right)^{i-1}{\bf B}{\bf e}_{k-i-1}+\sum_{i=2}^{k-2}\left({\bf A}+{\bf B}{\bf L}\right)^{i-1}{\bf K}{\bar \gamma}_{k-i}\right) \nonumber \\ 
	&-{\bf C}{\bf B}{\bf e}_{k-1}+\left({\bf A}_a-{\bf C}\left({\bf A}+{\bf B}{\bf L}\right){\bf K}\right){\bar {\bf \gamma}}_{k-1} \label{eqn:gamma_attack_apndx_2}.
\end{align}
Since we have assumed that the system started at $k=-\infty$, and $\left({\bf A}+{\bf B}{\bf L}\right)$ is strictly stable, we can say $\left({\bf A}+{\bf B}{\bf L}\right)^{k-2} \approx {\bf 0}$, and (\ref{eqn:gamma_attack_apndx_2}) will take the following form 
\begin{align}
	& \widetilde\gamma_k ={\bf w}_{a,k-1}+\left({\bf A}_a{\bf C}-{\bf C}\left({\bf A}+{\bf B}{\bf L}\right)\right) \cr 
	&\left(  \sum_{i=1}^{k-2}\left({\bf A}+{\bf B}{\bf L}\right)^{i-1}{\bf B}{\bf e}_{k-i-1}+\sum_{i=2}^{k-2}\left({\bf A}+{\bf B}{\bf L}\right)^{i-1}{\bf K}{\bar \gamma}_{k-i}\right) \cr 
	& -{\bf C}{\bf B}{\bf e}_{k-1}+\left({\bf A}_a-{\bf C}\left({\bf A}+{\bf B}{\bf L}\right){\bf K}\right){\bar {\bf \gamma}}_{k-1}.
	\label{eqn:gamma_attack_apndx_3}
\end{align}

Therefore,
\begin{align}
	&{\bf \mu}_{{\widetilde {\bf \gamma}_k}|\left\{\bar \gamma\right\}_1^{k-1},\left\{{\bf e}\right\}_1^{k-1}} =E\left[{\widetilde {\bf \gamma}_k}|{\bf z}_{k-1},{\bf {\hat x}}^F_{k-1|k-1},{\bf e}_{k-1}\right] \cr
	&={\bf A}_a{\bf z}_{k-1}-{\bf C}\left({\bf A}+ {\bf B}{\bf L} \right){\bf {\hat x}}^F_{k-1|k-1} - {\bf C}{\bf B}{\bf e}_{k-1}, \text{and} \\ 
	&{\bf \Sigma}_{{\widetilde {\bf \gamma}_k}|\left\{\bar \gamma\right\}_1^{k-1},\left\{{\bf e}\right\}_1^{k-1}}=cov\left({\widetilde {\bf \gamma}_k}|{\bf z}_{k-1},{\bf {\hat x}}^F_{k-1|k-1},{\bf e}_{k-1}\right)={\bf Q}_a. \cr
\end{align}
Furthermore, using (\ref{eqn:gamma_e}) we obtain $E\left[\gamma_k \right]=0$ and 
\begin{align}
	\gamma_k&={\bf y}_k-{\bf C}{\hat {\bf x}}_{k|k-1}={\bf C}\left({\bf x}_k-{\hat {\bf x}}_{k|k-1}\right)+{\bf v}_k \text{, and} \nonumber \\
	{\bf \Sigma}_\gamma &= E\left[\gamma_k\gamma_k^T\right] ={\bf C}{\bf P}{\bf C}^T+{\bf R}. \label{eqn:Sigma_gamma_apn_corr1p1}
\end{align}

\section{Proof of Corollary~\ref{cor:subopt_cusum} }
\label{apdx:subopt_cusum}
The covariance matrix ($E\left[{\widetilde \gamma}_k{\bf e}_{k-1}^T\right]$) between $\widetilde \gamma_k$ (\ref{eqn:gamma_e_attack}) and ${\bf e}_{k-1} $ is evaluated as, 
\begin{align}
	E\left[{\widetilde \gamma}_k{\bf e}_{k-1}^T\right] = E\left[-{\bf C}{\bf B}{\bf e}_{k-1}{\bf e}_{k-1}^T\right]=-{\bf C}{\bf B}{\bf \Sigma}_e, \label{eqn:cov_gamma_attack_ek_cor11}
\end{align}
since ${\bf e}_{k-1}$ is uncorrelated with ${\bf z}_k$ and ${\hat {\bf x}}_{k-1|k-1}^F$.

\section{Proof of Lemma~\ref{lemma:sigma_gammas}}
\label{apdx:sigma_sq_gamma_attack}
The variance of the innovation signal (${\bf \Sigma}_{\widetilde \gamma}$) when the system is under attack is derived in this section. Using (\ref{eqn:gamma_e_attack}), and applying the knowledge that ${\bf e}_{k-1}$ is uncorrelated with ${\bf z}_k$ and ${\hat {\bf x}}_{k-1|k-1}^F$, we get the following expression of ${\bf \Sigma}_{\widetilde \gamma}$,
\begin{align}
	&{\bf \Sigma}_{\widetilde \gamma}=E\left[{\widetilde \gamma}_k{\widetilde \gamma}_k^T\right]= E\left[{\bf z}_k{\bf z}_k^T\right]-{\bf C}\left({\bf A}+{\bf B}{\bf L}\right)E\left[{\hat {\bf x}}_{k-1|k-1}^F{\bf z}_k^T\right] \cr
	&-\left({\bf C}\left({\bf A}+{\bf B}{\bf L}\right)E\left[{\hat {\bf x}}_{k-1|k-1}^F{\bf z}_k^T\right]\right)^T+{\bf C}{\bf B}{\bf \Sigma}_e{\bf B}^T{\bf C}^T  \nonumber \\
	&+{\bf C}\left({\bf A}+{\bf B}{\bf L}\right)E\left[{\hat {\bf x}}_{k-1|k-1}^F\left({\hat {\bf x}}_{k-1|k-1}^{F}\right)^T\right]\left({\bf A}+{\bf B}{\bf L}\right)^T{\bf C}^T. \cr
	\label{eqn:Sigma_gamma_attack_stage_1}
\end{align}
We first derive the expressions of $E\left[{\hat {\bf x}}_{k-1|k-1}^F{\bf z}_k^T\right]$ (\ref{eqn:Exz_1_series}) and $E\left[{\hat {\bf x}}_{k-1|k-1}^F\left({\hat {\bf x}}_{k-1|k-1}^{F}\right)^T\right]$ (\ref{eqn:ExFxF_apn}), and then use them to get the final expression of ${\bf \Sigma}_{\widetilde \gamma}$ (\ref{eqn:sigma_gamma_sq_final_apndx}). $E\left[{\hat {\bf x}}_{k-1|k-1}^F{\bf z}_k^T\right]$ is calculated using (\ref{eqn:xF_k_k_1})-(\ref{eqn:gamma_k_attack}) and (\ref{eqn:add_ek}) as follows. First note that
\begin{align}
	&{\hat {\bf x}}_{k-1|k-1}^F={\bf K}{\bf z}_{k-1}+\mathcal{A}{\hat {\bf x}}_{k-2|k-2}^F +\left({\bf I}_n-{\bf K}{\bf C}\right){\bf B}{\bf e}_{k-2}, \nonumber \\
	&\text{where }\mathcal{A}=\left({\bf I}_n-{\bf K}{\bf C}\right)\left({\bf A}+{\bf B}{\bf L}\right).
	\label{eqn:xF_k_1_k_1}
\end{align}
We define ${\bf E}_{xz}\left(-k_0\right) \triangleq E\left[{\hat {\bf x}}_{k-k_0|k-k_0}^F{\bf z}_k^T\right]$, 
\begin{align}
	\begin{split}
		&=E\left[\left({\bf K}{\bf z}_{k-k_0}+\mathcal{A}{\hat {\bf x}}_{k-k_0-1|k-k_0-1}^F \right.\right.  \\
		&+\left.\left.\left({\bf I}_n-{\bf K}{\bf C}\right){\bf B}{\bf e}_{k-k_0-1}\right){\bf z}_{k}^T \right] \text{, [using (\ref{eqn:xF_k_1_k_1})]}  \cr
		&={\bf K}{\bf E}_{zz}\left(-k_0\right)+\mathcal{A}{\bf E}_{xz}\left(-k_0-1\right),
	\end{split}
	\label{eqn:Exz_k0}
\end{align}
where ${\bf e}_{k-k_0-1}$ and ${\bf z}_{k}$ are uncorrelated, and ${\bf E}_{zz}\left(-k_0\right)$ is evaluated as follows. 
\begin{align}
	&{\bf E}_{zz}\left(-k_0\right)={\bf E}_{zz}\left(k_0\right)=E\left[{\bf z}_{k}{\bf z}_{k-k_0}^T\right], \cr
	&{\bf E}_{zz}\left(-1\right)=E\left[{\bf A}_a{\bf z}_{k-1}{\bf z}_{k-1}^T+{\bf w}_{a,k-1}{\bf z}_{k-1}^T\right] ={\bf A}_a{\bf E}_{zz}\left(0\right), \cr
	&\text{because }{\bf w}_{a,k} \text{ and }{\bf z}_{k} \text{ are uncorrelated. Similarly,}\cr
	&{\bf E}_{zz}\left(-2\right)={\bf A}_a{\bf E}_{zz}\left(-1\right)={\bf A}_a^2{\bf E}_{zz}\left(0\right) \text{, and} \cr
	&{\bf E}_{zz}\left(-k_0\right) = {\bf A}_a^{k_0}{\bf E}_{zz}\left(0\right).
	\label{eqn:Ezz_k0}
\end{align}
The system matrix ${\bf A}_a$ is assumed to be strictly stable because the attacker will always try to generate fake observations which are bounded and will mimic the true observations to remain stealthy. For a strictly stable ${\bf A}_a$, 
\begin{align}
	&{\bf A}_a^{k_0} \rightarrow 0 \text{, as } k_0 \rightarrow \infty. \cr
	&\text{Therefore, }{\bf E}_{zz}\left(-k_0\right) \rightarrow 0 \text{, as } k_0 \rightarrow \infty.
	\label{eqn:Exa_infty}
\end{align}
Using (\ref{eqn:Exz_k0}) and (\ref{eqn:Ezz_k0}), we can write the expression of ${\bf E}_{xz}(-1)$ as
\begin{align}
	&{\bf E}_{xz}\left(-1\right)={\bf K}{\bf E}_{zz}\left(-1\right)+\mathcal{A}{\bf E}_{xz}\left(-2\right) \cr
	&={\bf K}{\bf A}_a^{}{\bf E}_{zz}\left(0\right)+\mathcal{A}\left({\bf K}{\bf E}_{zz}\left(-2\right)+\mathcal{A}{\bf E}_{xz}\left(-3\right) \right) \cr  
	& \text{[after replacing } {\bf E}_{xz}\left(-2\right) \text{ using (\ref{eqn:Exz_k0})]} \cr
	&={\bf K}{\bf A}_a^{}{\bf E}_{zz}\left(0\right)+\mathcal{A}{\bf K}{\bf A}_a^{2}{\bf E}_{zz}\left(0\right)+\mathcal{A}^2{\bf E}_{xz}\left(-3\right).
	\label{eqn:Exz_1}
\end{align}
Repeating the same technique, ${\bf E}_{xz}\left(-1\right)$ will take the following form,
\begin{equation}
	{\bf E}_{xz}\left(-1\right)= \sum_{i=0}^\infty\mathcal{A}^{i}{\bf K}{\bf C}_a{\bf A}_a^{i+1}{\bf E}_{x_a}\left(0\right){\bf C}_a^T.
	\label{eqn:Exz_1_series}
\end{equation}
${\bf E}_{xz}\left(-1\right)$ can be evaluated numerically by taking a large number of terms for the summation (\ref{eqn:Exz_1_series}), until the rest of the terms become negligible. 
${\bf E}_{x^Fx^F}(0) = E\left[{\hat {\bf x}}_{k-1|k-1}^F\left({\hat {\bf x}}_{k-1|k-1}^{F}\right)^T\right]$ is evaluated using (\ref{eqn:xF_k_1_k_1}) as
\begin{align}
	&{\bf E}_{x^Fx^F}(0)= {\bf K}E\left[{\bf z}_{k-1}{\bf z}_{k-1}^T\right]{\bf K}^T+\mathcal{A}E\left[{\hat {\bf x}}_{k-2|k-2}^F{\bf z}_{k-1}^T\right]{\bf K}^T \cr
	&+\left(\mathcal{A}E\left[{\hat {\bf x}}_{k-2|k-2}^F{\bf z}_{k-1}^T\right]{\bf K}^T\right)^T  \cr
	& +\mathcal{A}E\left[{\hat {\bf x}}_{k-2|k-2}^F\left({\hat {\bf x}}_{k-2|k-2}^{F}\right)^T\right]\mathcal{A}^T \nonumber \\
	&+\left({\bf I}_n-{\bf K}{\bf C}\right){\bf B}E\left[{\bf e}_{k-2}{\bf e}_{k-2}^T\right]{\bf B}^T\left({\bf I}_n-{\bf K}{\bf C}\right)^T.
	\label{eqn:ExFxF}
\end{align}
Therefore, ${\bf E}_{x^Fx^F}(0)$ is the solution to the following Lyapunov equation,
\begin{align}
	&\mathcal{A}{\bf E}_{x^Fx^F}(0)\mathcal{A}^T-{\bf E}_{x^Fx^F}(0)+{\bf K}{\bf E}_{zz}(0){\bf K}^T \cr
	&+\mathcal{A}{\bf E}_{xz}(-1){\bf K}^T+\left(\mathcal{A}{\bf E}_{xz}(-1){\bf K}^T\right)^T + \\
	&\left({\bf I}_n-{\bf K}{\bf C}\right){\bf B}{\bf \Sigma}_e{\bf B}^T\left({\bf I}_n-{\bf K}{\bf C}\right)^T = 0 \text{, [(\ref{eqn:Ezz_k0}) used]}. \nonumber
\end{align}
${\bf E}_{x^Fx^F}(0)$ is divided into two parts, ${\bf \Sigma}_{x^Fz}$ and ${\bf \Sigma}_{x^Fe}$ which are independent of the watermarking signal and the fake observations, respectively. ${\bf \Sigma}_{x^Fz}$ and ${\bf \Sigma}_{x^Fe}$ are the solution to the following Lyapunov equations, 
\begin{align}
	&\mathcal{A}{\bf \Sigma}_{x^Fz}\mathcal{A}^T-{\bf \Sigma}_{x^Fz}+{\bf K}{\bf E}_{zz}(0){\bf K}^T+\mathcal{A}{\bf E}_{xz}(-1){\bf K}^T \cr
	&+\left(\mathcal{A}{\bf E}_{xz}(-1){\bf K}^T\right)^T = 0 \text{,} \cr
	&\mathcal{A}{\bf \Sigma}_{x^Fe}\mathcal{A}^T-{\bf \Sigma}_{x^Fe}+\left({\bf I}_n-{\bf K}{\bf C}\right){\bf B}{\bf \Sigma}_e{\bf B}^T\left({\bf I}_n-{\bf K}{\bf C}\right)^T = 0 \text{,} \cr
	&\text{and }{\bf E}_{x^Fx^F}(0)={\bf \Sigma}_{x^Fz}+{\bf \Sigma}_{x^Fe}. 
	\label{eqn:ExFxF_apn}
\end{align}
Using (\ref{eqn:Ezz_k0}) and (\ref{eqn:ExFxF_apn}), we can rewrite the expression for ${\bf \Sigma}_{\widetilde \gamma} $ as,
\begin{align}
	{\bf \Sigma}_{\widetilde \gamma}&={\bf E}_{zz}(0)-{\bf C}({\bf A}+{\bf B{\bf L)}}{\bf E}_{xz}(-1) \cr
	&-\left[{\bf C}({\bf A}+{\bf B{\bf L)}}{\bf E}_{xz}(-1) \right]^T+{\bf C}{\bf B}{\bf \Sigma}_e{\bf B}^T{\bf C}^T  \cr
	&+{\bf C}({\bf A}+{\bf B}{\bf L}){\bf \Sigma}_{x^Fz}({\bf A}+{\bf B}{\bf L})^T{\bf C}^T \cr
	&+{\bf C}({\bf A}+{\bf B}{\bf L}){\bf \Sigma}_{x^Fe}({\bf A}+{\bf B}{\bf L})^T{\bf C}^T.
	\label{eqn:sigma_gamma_sq_final_apndx}
\end{align}


\section{Proof of Corollary~\ref{cor:Exzm1}}
\label{apdx:kld_corr} We can simplify ${\bf E}_{xz}\left(-1\right)$ with the assumption that both $\mathcal{A}$ and ${\bf A}_a$ are diagonalizable. 
If $\mathcal{A}$ and ${\bf A}_a$ are diagonalizable, then the $i$-th element of the expression for ${\bf E}_{xz}\left(-1\right)$, \ie, $\mathcal{A}^{i}{\bf K}{\bf A}_a^{i+1}{\bf E}_{zz}\left(0\right)$, will take the following form,
\begin{align}
	&{\bf U}_{ \mathcal{A}}{\bf \Sigma}_{ \mathcal{A}}^{i}{\bf U}_{ \mathcal{A}}^{-1}{\bf K}{\bf U}_a{\bf \Sigma}_a^{i}{\bf U}_a^{-1}{\bf A}_a{\bf E}_{zz}(0) \text{ [} \mathcal{A} \text{ and } {\bf A}_a \text{ replaced}  \nonumber \\
	& \text{by eigenvalue decompositions, (\ref{eqn:svdscriptA}) and (\ref{eqn:svdAa})]} \cr
	&={\bf U}_{ \mathcal{A}}{\bf \Sigma}_{ \mathcal{A}}^{i}{\bf T}{\bf \Sigma}_a^{i}{\bf U}_a^{-1}{\bf A}_a{\bf E}_{zz}(0) \text{, [} i=0,\ \cdots\ ,\infty \text{]} 
	\label{eqn:Ta}
\end{align}
where ${\bf T} = {\bf U}_{ \mathcal{A}}^{-1}{\bf K}{\bf U}_a$. ${\bf T}_a$ is defined as
\begin{equation}
	{\bf T}_a \triangleq \sum_{i=0}^\infty{\bf \Sigma}_{ \mathcal{A}}^{i}{\bf T}{\bf \Sigma}_a^{i}.
\end{equation}
The $jk$-th element of the matrix ${\bf T}_a$ will be as follows
\begin{equation}
	\left[{\bf T}_a\right]_{jk}=\sum_{i=0}^\infty\left[{\bf T}\right]_{jk}\lambda_{\mathcal{A},j}^{i}\lambda_{a,k}^{i} =\frac{\left[{\bf T}\right]_{jk}}{1-\lambda_{\mathcal{A},j}\lambda_{a,k}}
	\label{eqn:Ta_jk}
\end{equation}
where $[.]_{jk}$ denotes the $j$-th row and $k$-th column element of a matrix. $\lambda_{\mathcal{A},j}$ and $\lambda_{a,k}$ are the $j$-th and $k$-th diagonal element of the diagonal matrices ${\bf \Sigma}_{\mathcal{A}}$ and ${\bf \Sigma}_{{a}}$ respectively. We assume $\mathcal{A}$ and ${\bf A}_a$ to be strictly stable, therefore, $|{\bf \lambda}_{\mathcal{A},j}|<1$ and $|{\bf \lambda}_{{a},k}|<1$. $|.|$ denotes the absolute value of a scalar.
Using (\ref{eqn:Ta_jk}), we can write
\begin{align}
	{\bf E}_{xz}\left(-1\right)={\bf U}_{\mathcal{A}}{\bf T}_a{\bf U}_a^{-1}{\bf A}_a{\bf E}_{zz}(0). 
	\label{eqn:Exz_1_apn}
\end{align}

\section{Proof of Theorem~\ref{th:th_kld}}
\label{apdx:kld_mimo}
This section provides the proof of the Theorem~\ref{th:th_kld} under the optimal CUSUM and sub-optimal CUSUM test. The KLDs for both the cases are derived using the general expression of KLD between two multivariate normal distributions given in \cite{Duchi1001}.
Using (\ref{eqn:f_mean_attack_dept}), (\ref{eqn:f_gamma_attack_dept}) and (\ref{eqn:gamma_attack_apndx}), and considering that ${\bf e}_{k}$ and ${\bf w}_{a,k}$ are uncorrelated with ${\bf z}_{k}$ and ${\bf {\hat x}}^F_{k|k}$, and also with each other, we can write, 
\begin{align}
	{\bf \Sigma}_{\widetilde \gamma}={\bf Q}_a+E\left[{\bf \mu}_{{\widetilde {\bf \gamma}_k}|\left\{\bar \gamma\right\}_1^{k-1},\left\{{\bf e}\right\}_1^{k-1}}{\bf \mu}_{{\widetilde {\bf \gamma}_k}|\left\{\bar \gamma\right\}_1^{k-1},\left\{{\bf e}\right\}_1^{k-1}}^T\right].
	\label{eqn:sigma_sq_gamma_dep_apn}
\end{align}
The expected KLD $E\left[ D\left(f_{{\widetilde {\bf \gamma}_k}},f_{{ {\bf \gamma}_k}} |\left\{\bar \gamma\right\}_1^{k-1},\left\{{\bf e}\right\}_1^{k-1}\right)\right]$ under the optimal CUSUM test is derived as follows using \cite{Duchi1001}, see (\ref{eqn:opt_kld_apn}).

\begin{table*}\centering
	\begin{tabular}{  m{10cm} }\hline
		{\begin{flalign}
				&E\left[\frac{1}{2}\left(tr\left({\bf \Sigma}_{{\bf \gamma}}^{-1}{\bf \Sigma}_{{\widetilde {\bf \gamma}_k}|\left\{\bar \gamma\right\}_1^{k-1},\left\{{\bf e}\right\}_1^{k-1}}\right) - m +{\bf \mu}_{{\widetilde {\bf \gamma}_k}|\left\{\bar \gamma\right\}_1^{k-1},\left\{{\bf e}\right\}_1^{k-1}} ^T{\bf \Sigma}_{{\bf \gamma}}^{-1}{\bf \mu}_{{\widetilde {\bf \gamma}_k}|\left\{\bar \gamma\right\}_1^{k-1},\left\{{\bf e}\right\}_1^{k-1}} -\log\frac{\left| {\bf \Sigma}_{{\widetilde {\bf \gamma}_k}|\left\{\bar \gamma\right\}_1^{k-1},\left\{{\bf e}\right\}_1^{k-1}} \right|}{\left|{\bf \Sigma}_{{\bf \gamma}} \right|}\right)\right] \nonumber \\
				&=\frac{1}{2}\left(-m + tr\left({\bf \Sigma}_{{\bf \gamma}}^{-1}{\bf \Sigma}_{{\widetilde {\bf \gamma}_k}|\left\{\bar \gamma\right\}_1^{k-1},\left\{{\bf e}\right\}_1^{k-1}} +{\bf \Sigma}_{{\bf \gamma}}^{-1}E\left[{\bf \mu}_{{\widetilde {\bf \gamma}_k}|\left\{\bar \gamma\right\}_1^{k-1},\left\{{\bf e}\right\}_1^{k-1}}{\bf \mu}_{{\widetilde {\bf \gamma}_k}|\left\{\bar \gamma\right\}_1^{k-1},\left\{{\bf e}\right\}_1^{k-1}} ^T\right]\right)  -\log\frac{\left| {\bf \Sigma}_{{\widetilde {\bf \gamma}_k}|\left\{\bar \gamma\right\}_1^{k-1},\left\{{\bf e}\right\}_1^{k-1}} \right|}{\left|{\bf \Sigma}_{{\bf \gamma}} \right|}\right) \nonumber \\
				& =\frac{1}{2}\left\{tr\left({\bf \Sigma}_\gamma^{-1} {\bf \Sigma}_{\widetilde \gamma}\right) -m - \log\frac{\mid{\bf Q}_a\mid}{\mid{\bf \Sigma}_{\gamma}\mid}\right\} \text{, [using (\ref{eqn:sigma_sq_gamma_dep_apn}) \& (\ref{eqn:f_gamma_attack_dept})]}. \label{eqn:opt_kld_apn}
		\end{flalign}}\\
		\hline \end{tabular}\label{NS_eqt}\end{table*}
Similarly, the KLD $D\left(f_{{\widetilde {\bf \gamma}_k},{\bf e}_{k-1}},f_{{ {\bf \gamma}_k},{\bf e}_{k-1}}\right)$ under the sub-optimal CUSUM test will take the following form \cite{Duchi1001},
\begin{align}
	\frac{1}{2}\left( \log \frac{\left|{\bf \Sigma}_{\gamma_e}\right|}{\left|{\bf \Sigma}_{\widetilde\gamma_e}\right|} -p-m + tr\left({\bf \Sigma}_{\gamma_e}^{-1}{\bf \Sigma}_{\widetilde\gamma_e} \right) \right).
	\label{eqn:kld_full_apn}
\end{align}
The term $\log \frac{\left|{\bf \Sigma}_{\gamma_e}\right|}{\left|{\bf \Sigma}_{\widetilde\gamma_e}\right|}$ is evaluated as follows,
\begin{align}
	&\left|{\bf \Sigma}_{\gamma_e}\right|=\left|{\bf \Sigma}_{e}\right|\left|{\bf \Sigma}_{\gamma}\right| \text{, [using (\ref{eqn:sgima_sq_gamma_e})]}  \label{eqn:det_Sigma_gamma_e} \ \\
	&\left|{\bf \Sigma}_{\widetilde\gamma_e}\right|=\left|{\bf \Sigma}_{e}\right|\left|{\bf \Sigma}_{\widetilde \gamma}-{\bf C}{\bf B}{\bf \Sigma}_e{\bf B}^T{\bf C}^T\right| \text{, [using (\ref{eqn:sgima_sq_gamma_e_attack})]}. \label{eqn:det_Sigma_gamma_e_attack} \\
	&\text{Therefore, }\log \frac{\left|{\bf \Sigma}_{\gamma_e}\right|}{\left|{\bf \Sigma}_{\widetilde\gamma_e}\right|} = -\log\frac{\left|{\bf \Sigma}_{\widetilde \gamma}-{\bf C}{\bf B}{\bf \Sigma}_e{\bf B}^T{\bf C}^T\right|}{\left|{\bf \Sigma}_{\gamma}\right|}. 
	\label{eqn:log_sigam_apn}
\end{align}
The term $tr\left({\bf \Sigma}_{\gamma_e}^{-1}{\bf \Sigma}_{\widetilde\gamma_e} \right)$ is evaluated using (\ref{eqn:sgima_sq_gamma_e}) and (\ref{eqn:sgima_sq_gamma_e_attack}) as,
\begin{equation}
	tr\left({\bf \Sigma}_{\gamma_e}^{-1}{\bf \Sigma}_{\widetilde\gamma_e} \right) = tr\left({\bf \Sigma}_{\gamma}^{-1}{\bf \Sigma}_{\widetilde\gamma}+{\bf \Sigma}_{e}^{-1}{\bf \Sigma}_{e} \right)=tr\left({\bf \Sigma}_{\gamma}^{-1}{\bf \Sigma}_{\widetilde\gamma} \right)+p
	\label{eqn:tr_sigma_sq_gamma_apndx}
\end{equation} 
Applying (\ref{eqn:log_sigam_apn}) and (\ref{eqn:tr_sigma_sq_gamma_apndx}) in (\ref{eqn:kld_full_apn}), we get the final expression of the KLD $D\left(f_{{\widetilde {\bf \gamma}_k},{\bf e}_{k-1}},f_{{ {\bf \gamma}_k},{\bf e}_{k-1}}\right)$ under the sup-optimal CUSUM test as  
\begin{align}
	\frac{1}{2}\left\{tr\left({\bf \Sigma}_\gamma^{-1} {\bf \Sigma}_{\widetilde \gamma}\right) -m - \log\frac{\mid{{\bf \Sigma}_{\widetilde \gamma}}-{\bf C}{\bf B}{\bf \Sigma}_e{\bf B}^T{\bf C}^T\mid}{\mid{\bf \Sigma}_{\gamma}\mid}\right\}. \label{eqn:subopt_kld_apn}
\end{align}

\section{Proof of Lemma~\ref{le:kld_miso} }
\label{apdx:kld_miso}
This section provides the derivation of the expression of $\sigma_{\widetilde \gamma}^2$ for the MISO system. 
The model parameters of the fake measurement generation system (\ref{eqn:hidden_states_main})) for the MISO system will be as follows. 
\begin{equation}
	{\bf A}_a = \rho \text{, } {\bf Q}_a = \left(1 - \rho^2 \right)\sigma_z^2, \text{and } {\bf E}_{zz}(0)=\sigma_z^2.
	\label{eqn:fake_meas_params}
\end{equation}
To evaluate $\sigma_{\widetilde \gamma}^2$, we derive the expression for ${\bf E}_{xz}(-1)$ for a MISO system using (\ref{eqn:Exz1_original}) as
\begin{align}
	&{\bf E}_{xz}(-1)= \sum_{i=0}^\infty\mathcal{A}^{i}{\bf K}{\bf A}_a^{i+1}{\bf E}_{zz} (0) \nonumber \\
	&=\sum_{i=0}^\infty\mathcal{A}^{i}{\bf K}\rho^{i+1}\sigma_z^2  \text{, [} {\bf E}_{zz}\left(0\right)=\sigma_z^2,\ {\bf A}_a=\rho \text{]} \nonumber \\
	& = \left[{\bf I}_n-\rho{\mathcal A}\right]^{-1}{\bf K}\rho\sigma_z^2 \text{, [}  {\mathcal A} \text{ is strictly stable, }\rho <1 \text{]}.
	\label{eqn:Exz_miso_apn}
\end{align}
$\sigma_{\widetilde \gamma}^2$ will be as follows,
\begin{align}
	&\sigma_{\widetilde \gamma}^2=\sigma_z^2-2{\bf C}\left({\bf A}+{\bf B}{\bf L}\right){\bf E}_{xz}(-1)+{\bf C}{\bf B}{\bf \Sigma}_e{\bf B}^T{\bf C}^T\cr 
	&+{\bf C}\left({\bf A} +{\bf B}{\bf L}\right){\bf \Sigma}_{x^Fz}\left({\bf A}+{\bf B}{\bf L}\right)^T{\bf C}^T \cr
	&+{\bf C}\left({\bf A}+{\bf B}{\bf L}\right){\bf \Sigma}_{x^Fe}\left({\bf A}+{\bf B}{\bf L}\right)^T{\bf C}^T \text{ [using (\ref{eqn:sigma_gamma_attack})]},
	\label{eqn:sigma_gamma_sq_attack_apm_part1}
\end{align}
where ${\bf \Sigma}_{x^Fz}$ and ${\bf \Sigma}_{x^Fe}$ are derived from (\ref{eqn:ExFxF_th1p1}) and (\ref{eqn:ExFxF_th1p2}) respectively as follows. 
\begin{equation}
	{\bf \Sigma}_{x^Fz}={\bf \Sigma}^z_{x^F}\sigma_z^2
	\label{eqn:Sigma_xfz_miso_apn}
\end{equation}
where ${\bf \Sigma}^z_{x^F}$ is the solution to the following Lyapunov equation, 
\begin{align}
	&{\bf \mathcal A}{\bf \Sigma}^z_{x^F}{\bf \mathcal A}^T-{\bf \Sigma}^z_{x^F}+ {\bf K}{\bf K}^T+{\bf \mathcal{A}}\left[{\bf I}_n-\rho{\bf \mathcal A}\right]^{-1}{\bf K}{\bf K}^T\rho \cr
	&+\left[{\bf \mathcal{A}}\left[{\bf I}_n-\rho{\bf \mathcal A}\right]^{-1}{\bf K}{\bf K}^T\rho\right]^T=0.
\end{align}
${\bf \Sigma}_{x^Fe}$  is the solution to the following Lyapunov equation,
\begin{align}
	{\bf \mathcal A}{\bf \Sigma}_{x^Fe}{\bf \mathcal A}^T-{\bf \Sigma}_{x^Fe}+\left({\bf I}_n - {\bf K}{\bf C}\right){\bf B}{\bf \Sigma}_e{\bf B^T}\left({\bf I}_n - {\bf K}{\bf C}\right)^T =0.
\end{align}
Using (\ref{eqn:Exz_miso_apn}) and (\ref{eqn:Sigma_xfz_miso_apn}), the expression for $\sigma_{\widetilde \gamma}^2$ (\ref{eqn:sigma_gamma_sq_attack_apm_part1}) can be rearranged as follows. 
\begin{align}
	& \sigma_{\gamma}^2=\left(1-2{\bf C}\left({\bf A}+{\bf B}{\bf L}\right)\left({\bf I}_n-\rho\mathcal{A}\right)^{-1}{\bf K} \rho \right. \cr
	&\left. + {\bf C}\left({\bf A}+{\bf B}{\bf L}\right){\bf \Sigma}^z_{x^F}\left({\bf A}+{\bf B}{\bf L}\right)^T{\bf C}^T\right)\sigma_z^2 \cr
	&+\left({\bf C}\left({\bf A}+{\bf B}{\bf L}\right){\bf \Sigma}_{xe}\left({\bf A}+{\bf B}{\bf L}\right)^T{\bf C}^T +{\bf C}{\bf B}{\bf \Sigma}_e{\bf B}^T{\bf C}^T \right) \cr
	&=M_z\sigma_z^2+M_t
	\label{eqn:sigma_gamma_sq_attack_apm_part2}
\end{align}
The scalar quantity $M_t$ can be rearranged as follows.
\begin{align}
	&M_t = \left(\sum_{t=0}^\infty{\bf C}\left({\bf A}+{\bf B}{\bf L}\right) \mathcal{A}^t\left({\bf I}_n - {\bf K}{\bf C}\right){\bf B}{\bf \Sigma}_e{\bf B^T}\left({\bf I}_n - {\bf K}{\bf C}\right)^T \right. \cr
	& \left. \left[\mathcal{A}^{T}\right]^{t}\left({\bf A}+{\bf B}{\bf L}\right)^T{\bf C}^T \right)+{\bf C}{\bf B}{\bf \Sigma}_e{\bf B}^T{\bf C}^T  \cr
	&=\text{tr}\left(\sum_{t=0}^\infty{\bf B^T}\left({\bf I}_n - {\bf K}{\bf C}\right)^T \left[\mathcal{A}^{T}\right]^{t}\left({\bf A}+{\bf B}{\bf L}\right)^T{\bf C}^T{\bf C}\left({\bf A}+{\bf B}{\bf L}\right) \right. \nonumber \\
	& \left. \mathcal{A}^t\left({\bf I}_n - {\bf K}{\bf C}\right){\bf B}{\bf \Sigma}_e +{\bf B}^T{\bf C}^T{\bf C}{\bf B}{\bf \Sigma}_e \right) = \text{tr}\left(M_e{\bf \Sigma}_e \right), \\
	&\text{where }M_e = {\bf B^T}\left({\bf I}_n - {\bf K}{\bf C}\right)^T{\bf \Sigma}^e_{x^F}\left({\bf I}_n - {\bf K}{\bf C}\right){\bf B}+{\bf B}^T{\bf C}^T{\bf C}{\bf B}.
\end{align}
${\bf \Sigma}^e_{x^F}$ is the solution to the following Lyapunov equation, 
\begin{equation}
	\mathcal{A}^{T}{\bf \Sigma}^e_{x^F}\mathcal{A}-{\bf \Sigma}^e_{x^F}+\left({\bf A}+{\bf B}{\bf L}\right)^T{\bf C}^T{\bf C}\left({\bf A}+{\bf B}{\bf L}\right)=0.
	\label{eqn:Me_apn}
\end{equation}
Finally, we can write $\sigma_{\widetilde \gamma}^2$ as
\begin{equation}
	\sigma_{\widetilde \gamma}^2 = M_z\sigma_z^2+\text{tr}\left(M_e{\bf \Sigma}_e\right).
\end{equation}

\section{Proof of Theorem~\ref{th:opt_diagonal_Sigma_e}}
\label{apdx:optimization}
The covariance matrix of the watermarking signal is decomposed using eigenvalue decomposition as follows,
\begin{equation}
	{\bf \Sigma}_e={\bf V}_e{\bf \Lambda}_e{\bf V}_e^T
	\label{eqn:sigma_e_eig_apn}
\end{equation}
where ${\bf V}_e$ and ${\bf \Lambda}_e$ are the eigenvector matrix and the diagonal eigenvalue matrix. In this section, we will prove that KLD is convex with respect to the elements of ${\bf \Lambda}_e$ for a fixed ${\bf V}_e$. We formulate the optimization problem as follows. 
\begin{align}
	\max_{{\bf \Lambda}_e} &\  f\left({\bf \Lambda}_e\right) = \  E\left[ D\left(f_{{\widetilde {\bf \gamma}_k}},f_{{ {\bf \gamma}_k}} |\left\{\bar \gamma\right\}_1^{k-1},\left\{{\bf e}\right\}_1^{k-1}\right)\right] \text{or}  \nonumber \\
	\max_{{\bf \Lambda}_e} &\  f\left({\bf \Lambda}_e\right) = \  D\left(f_{\widetilde \gamma_k,{\bf e}_{k-1}},f_{ \gamma_k,{\bf e}_{k-1}}\right) \\
	\textrm{s.t.}\  & \Delta LQG \le J  \label{eqn:const_lambda_apn}\\
	\textrm{and}\ & {\bf \lambda}_{e,i} \ge 0, \forall i. \label{eqn:const_lambda_i_apn}
\end{align}
The proof for the optimal CUSUM case is as follows. 

Observing (\ref{eqn:opt_kld}) and (\ref{eqn:sigma_gamma_attack}), we can say that maximizing the expected KLD with respect to ${\bf \Sigma}_e$ is the same as maximizing the following portion of the expected KLD expression which is only dependent on ${\bf \Sigma}_e$. 
\begin{align}
	f\left({\bf \Sigma}_e\right)={\bf C}({\bf A}+{\bf B}{\bf L}){\bf \Sigma}_{x^Fe}({\bf A}+{\bf B}{\bf L})^T{\bf C}^T+{\bf C}{\bf B}{\bf \Sigma}_e{\bf B}^T{\bf C}^T 
	\label{eqn:fun_opt_d_cusum_apn}
\end{align}
where ${\bf \Sigma}_{x^Fe}$ is given by (\ref{eqn:ExFxF_th1p2}). Putting the solution of (\ref{eqn:ExFxF_th1p2}) in (\ref{eqn:fun_opt_d_cusum_apn}), we get,
\begin{align}
	&f\left({\bf \Sigma}_e\right)={\bf C}\left({\bf A}+{\bf B}{\bf L} \right)\left(\sum_{t=0}^\infty{\cal A}^t\left({\bf I}_n - {\bf K}{\bf C}\right){\bf B}{\bf \Sigma}_e{\bf B}^T  \right. \nonumber \\
	& \left. \left({\bf I}_n - {\bf K}{\bf C}\right)^T\left[{\cal A}^T\right]^t \right)\left({\bf A}+{\bf B}{\bf L} \right)^T+{\bf C}{\bf B}{\bf \Sigma}_e{\bf B}^T{\bf C}^T  \nonumber \\
	&=tr\left(\left({\bf B}^T\left({\bf I}_n - {\bf K}{\bf C}\right)^T{\cal L}_e\left({\bf I}_n - {\bf K}{\bf C}\right){\bf B}+{\bf B}^T{\bf C}^T{\bf C}{\bf B}\right){\bf \Sigma}_e\right) \nonumber \\
	& = tr\left({\bf H}_{KLD}{\bf \Sigma}_e \right), 
	\label{eqn:fun_opt_d_cusum_apn_2}
\end{align}
where ${\cal L}_e$ is the solution to the following Lyapunov equation 
\begin{align}
	&{\cal A}^T{\cal L}_e{\cal A}-{\cal L}_e+\left({\bf A}+{\bf B}{\bf L} \right)^T{\bf C}^T{\bf C}\left({\bf A}+{\bf B}{\bf L} \right)=0, \text{and} 
	\label{eqn:cal_L_e_apn} \\
	& {\bf H}_{KLD} = {\bf B}^T \left({\bf I}_n - {\bf K}{\bf C}\right)^T{\cal L}_e\left({\bf I}_n - {\bf K}{\bf C}\right){\bf B}+{\bf B}^T{\bf C}^T{\bf C}{\bf B}.
	\label{eqn:Hkld_apn}
\end{align}
Using (\ref{eqn:fun_opt_d_cusum_apn_2}) and (\ref{eqn:sigma_e_eig_apn}), we can rewrite the cost function as follows 
\begin{equation}
	f\left({\bf \Lambda}_e\right)=tr\left({\bf V}_e^T {\bf H}_{KLD}{\bf V}_e{\bf \Lambda}_e\right)
\end{equation}
which represents a line in the $p$ dimensional hyperplane. Therefore, the cost function is convex in nature.

The proof for the sub-optimal CUSUM case is as follows. 
We have replaced all the ${\bf B}$ matrices by ${\bf B}_e$ where ${\bf B}_e = {\bf B}{\bf V}_e$ and ${\bf \Sigma}_e$ by ${\bf \Lambda}_e$ to keep the structure of the KLD and $\sigma_{\widetilde \gamma}^2$ expressions as (\ref{eqn:subopt_kld_miso}) and (\ref{eqn:sigma_sq_attack_th}) respectively. 
\begin{align}
	&f\left({\bf \Lambda}_e\right)=\frac{1}{2}\left(\frac{M_z{{\bf \sigma}_z^2}+\sum_{i=1}^n\left[{\bf M}_{e\lambda}\right]_{ii}\lambda_{e,i}}{{\bf \sigma}_\gamma^2}\right) \cr
	&-\frac{1}{2}\log\left(\frac{M_z{{\bf \sigma}_z^2}+\sum_{i=1}^n\left[{\bf M}_{em}\right]_{ii}\lambda_{e,i}}{{\bf \sigma}_\gamma^2}\right)
	\label{eqn:f_lambda_e} \\
	&\text{where }{\bf M}_{em}= {\bf B}_e^T\left({\bf I}_n - {\bf K}{\bf C}\right)^T{\bf \Sigma}^e_{x^F}\left({\bf I}_n- {\bf K}{\bf C}\right){\bf B}_e \text{, and}  \cr
	&{\bf M}_{e\lambda} = {\bf B}_e^T\left({\bf I}_n - {\bf K}{\bf C}\right)^T{\bf \Sigma}^e_{x^F}\left({\bf I}_n - {\bf K}{\bf C}\right){\bf B}_e+{\bf B}_e^T{\bf C}^T{\bf C}{\bf B}_e. \cr
	\label{eqn:M2_lambda}
\end{align}
The ${\bf \Sigma}^e_{x^F}$ is the same as in (\ref{eqn:Me_apn}).
The first derivative of the cost function with respect to the $j$-th eigenvalue $\lambda_{e,j}$ is as follows, 
\begin{align}
	&\frac{\partial }{\partial \lambda_{e,j}}f\left({\bf \Lambda}_e\right)=\frac{1}{2\sigma_\gamma^2}\left[{\bf M}_{e\lambda}\right]_{jj} \cr
	&-\frac{1}{2}\frac{1}{M_z{{\bf \sigma}_z^2}+\sum_{i=1}^n\left[{\bf M}_{em}\right]_{ii}\lambda_{e,i}}\left[{\bf M}_{em}\right]_{jj}.
	\label{eqn:first_derivative_f}
\end{align}
The second derivative of the cost function is as follows, 
\begin{align}
	&\frac{\partial }{\partial \lambda_{e,i}}\frac{\partial }{\partial \lambda_{e,j}}f\left({\bf \Lambda}_e\right)=\frac{1}{2}\left[{\bf M}_{em}\right]_{ii}\left[{\bf M}_{em}\right]_{jj}t_f^2 \text{, and} \cr
	&t_f=\frac{1}{M_z{{\bf \sigma}_z^2}+\sum_{i=1}^n\left[{\bf M}_{em}\right]_{ii}\lambda_{e,i}}
	\label{eqn:Hs_ij}
\end{align}
where $\frac{\partial }{\partial \lambda_{e,i}}\frac{\partial }{\partial \lambda_{e,j}}f\left({\bf \Lambda}_e\right)$ is the $ij$-th element of the Hessian matrix ${\bf H}_s=\triangledown_{{\bf \Lambda}_e}^2f\left({\bf \Lambda}_e\right)$. From (\ref{eqn:Hs_ij}), it is clear that each column of ${\bf H}_s$ is linearly dependent on any other column of the matrix. This means that we have all eigenvalues except one to be zero. Therefore, determinants of all the principle minors of ${\bf H}_s$ are zero. Also, the diagonal elements of ${\bf H}_s$ are non-zero. So, we can conclude that KLD is convex in ${\bf \Lambda}_e$. 

Since the cost function under both the tests are convex, the optimum ${\bf \Lambda}_e$, which maximizes the expected KLD or the KLD, will be on one of the vertices of the feasible region provided by (\ref{eqn:const_lambda_apn}) and (\ref{eqn:const_lambda_i_apn}). That is possible when the optimum ${\bf \Lambda}_e$ contains only one non-zero element. This property of the convex function over a polyhedron set can be proved using Jensen's inequality.

\section{Optimization algorithm}
\label{apdx:opt_algo}

The Lagrangian and it's first and second derivatives for the MISO system are given as follows. We multiply the cost function by -1 to convert the optimization problem into a minimization one.

For the optimal CUSUM test, using (\ref{eqn:H_KLD_apn}) and (\ref{eqn:deltaLQG}) the Lagrangian can be written in the following form
\begin{align}
	L\left({\bf v}_\lambda,\mu\right)= - {\bf v}_\lambda^T{\bf H}_{KLD}{\bf v}_\lambda+\mu\left({\bf v}_\lambda^T{\bf H}{\bf v}_\lambda-J\right). 
	\label{enq:L_v_opt}
\end{align}
The first derivatives of $L\left({\bf v}_\lambda,\mu\right)$ with respect to ${\bf v}_\lambda$ and $\mu$ are  
\begin{align}
	&\nabla_{{\bf v}_\lambda}L(.)={\bf C}_c{\bf v}_\lambda \text{, and} \\
	&\frac{\partial}{\partial \mu}L(.)={\bf v}_{\lambda}^T{\bf H}{\bf v}_{\lambda}-J \\
	\text{where }&{\bf C}_c = -2{\bf H}_{KLD}+2\mu{\bf H}.
\end{align}
The Hessian matrix of $L(.)$ with respect to ${\bf v}_{\lambda}$ is as follows,
\begin{align}
	{\bf H}_s=\nabla_{{\bf v}_\lambda}^2L(.)={\bf C}_{c}^T.
\end{align}

For the sub-optimal CUSUM test, we form the Lagrangian using (\ref{eqn:subopt_kld_miso}), (\ref{eqn:M2}), and (\ref{eqn:deltaLQG}) for the KLD and $\Delta LQG$ respectively as follows
\begin{align}
	&L\left({\bf v}_\lambda,\mu\right)=-\frac{1}{2}\left(\frac{M_z\sigma_z^2+{\bf v}_{\lambda}^T{\bf M}_{ev}{\bf v}_{\lambda}+{\bf v}_{\lambda}^T{\bf B}^T{\bf C}^T{\bf C}{\bf B}{\bf v}_{\lambda}  }{\sigma_\gamma^2}\right) \ \cr
	&-\frac{1}{2}+\frac{1}{2}\log\left(M_z\sigma_z^2+{\bf v}_{\lambda}^T{\bf M}_{ev}{\bf v}_{\lambda}\right)-\frac{1}{2}\log\left(\sigma_\gamma^2\right) \nonumber \\
	&+\mu\left({\bf v}_{\lambda}^T{\bf H}{\bf v}_{\lambda} - J \right), \label{enq:L_v}
\end{align}
where ${\bf M}_{ev}$ is the first part of the right hand side of (\ref{eqn:M2}), \ie, ${\bf M}_{ev}= {\bf B^T}\left({\bf I}_n - {\bf K}{\bf C}\right)^T{\bf \Sigma}^e_{x^F}\left({\bf I}_n - {\bf K}{\bf C}\right){\bf B}$.
The first derivatives of $L\left({\bf v}_\lambda,\mu\right)$ with respect to ${\bf v}_\lambda$ and $\mu$ are  
\begin{align}
	&\nabla_{{\bf v}_\lambda}L(.)=-\frac{1}{\sigma_\gamma^2}\left({\bf M}_{ev}{\bf v}_{\lambda}+{\bf B}^T{\bf C}^T{\bf C}{\bf B}{\bf v}_{\lambda}\right) \cr 
	&+\frac{{\bf M}_{ev}{\bf v}_{\lambda}}{M_z\sigma_z^2+{\bf v}_{\lambda}^T{\bf M}_{ev}{\bf v}_{\lambda}}+2\mu{\bf H}{\bf v}_{\lambda} ={\bf C}_c{\bf v}_{\lambda} \text{, and} \label{eqn:Delta_v_L} \\
	&\frac{\partial}{\partial \mu}L(.)={\bf v}_{\lambda}^T{\bf H}{\bf v}_{\lambda}-J,
\end{align}
\begin{align}
	&\text{where } {\bf C}_c = {\bf C}_{ca}+{\bf v}_{\lambda}^T{\bf M}_{ev}{\bf v}_{\lambda}{\bf C}_{cb}, \ \\
	& {\bf C}_{ca}=\left(1-\frac{{\bf M}_{z}\sigma_z^2}{\sigma_\gamma^2}\right){\bf M}_{ev}-\frac{{\bf M}_{z}\sigma_z^2}{\sigma_\gamma^2}{\bf B}^T{\bf C}^T{\bf C}{\bf B}+2\mu{\bf M}_{z}\sigma_z^2{\bf H}, \ \\
	&\text{and } {\bf C}_{cb} = 2\mu{\bf H}-\frac{1}{\sigma_\gamma^2}{\bf M}_{ev}-\frac{1}{\sigma_\gamma^2}{\bf B}^T{\bf C}^T{\bf C}{\bf B}.
\end{align}
The Hessian matrix of $L(.)$ with respect to ${\bf v}_{\lambda}$ is as follows
\begin{align}
	{\bf H}_s=\nabla_{{\bf v}_\lambda}^2L(.)={\bf C}_{ca}^T+2{\bf M}_{ev}{\bf v}_{\lambda}{\bf v}_{\lambda}^T{\bf C}_{cb}.
	\label{eqn:hessian}
\end{align}
A primal-dual approach to find the optimum ${\bf \Sigma}_e$ is provided in Algorithm~\ref{algo:opt_e}.
\begin{algorithm}[h!]
	\begin{algorithmic}
		\STATE Initialize: $s_0$, $K_{\mu,0}$, $max\_iteration$, and $\mu = 0$.
		\FOR{$k = 1: max\_iteration$}
		\STATE  Find the best solution ${\bf v}_{temp}^*$ for the set of equations, $\nabla_{{\bf v}_\lambda}L(.)=0$ and $\frac{\partial}{\partial \mu}L(.)=0$. 
		\IF {${\bf v}_{temp}^{T*}{\bf H}{\bf v}_{temp}^{*} - J \ne 0$}
		\STATE	$\mu \leftarrow \mu+s_k	\frac{\partial}{\partial \mu}L(.)$ 
		\ELSE
		\IF {${\bf H}_s \ge 0$} 
		\STATE ${\bf v}_{\lambda}^* \leftarrow {\bf v}_{temp}^*$ 
		\STATE break
		\ELSE
		\STATE $\mu \leftarrow \mu+ K_{\mu,k} \left(-\frac{\partial}{\partial \mu}L(.)  \right)$
		\ENDIF
		\ENDIF
		\ENDFOR
		\STATE ${\bf \Sigma}_e = {\bf v}_{\lambda}^*\left[{\bf v}_{\lambda}^*\right]^T$
	\end{algorithmic}
	\caption{To find optimum ${\bf \Sigma}_e$}
	\label{algo:opt_e}
\end{algorithm}
The step sizes ($s_k$, $K_{\mu,k}$) can be derived at every step using the backtracking algorithm \cite{boyd2004convex} which ensures the convergence to some local optima since the Hessian matrices under both the tests are indefinite matrices.


\section{System Parameters}
\label{apdx:system_param}
For both the systems, $ARL_h = 1000$. \\
\textbf{System-A parameters}:
\begin{align*}
	{\bf A} &=\begin{bmatrix}0.75 & 0.2 \\0.2 & 1.0 \end{bmatrix}           &  {\bf B} &=\begin{bmatrix}0.9 & 0.5 \\0.1 & 1.2 \end{bmatrix}              &  {\bf C}&=\begin{bmatrix}1.0 & -1.0  \end{bmatrix} \\
	{\bf Q} &=diag\begin{bmatrix}1 & 1  \end{bmatrix}           &  {\bf R} &=1              &  {\bf W}&=diag\begin{bmatrix}1 & 2  \end{bmatrix} \\
	{\bf U} &=diag\begin{bmatrix}0.4 & 0.7  \end{bmatrix}           &  \sigma_z^2 &=10             &  \rho&=0.5 
\end{align*}
\textbf{System-B parameters}:
\begin{align*}
	{\bf A} =\begin{bmatrix}0.968 &0&0.082 &0 \\ 0&0.978&0&0.064 \\ 0&0&0.917&0 \\ 0&0&0&0.935 \end{bmatrix} &  {\bf B} &=\begin{bmatrix} 0.164&0.004 \\0.002&0.124 \\ 0&0.092 \\ 0.060&0  \end{bmatrix} 
\end{align*}
\begin{align*}
	{\bf C} &= \begin{bmatrix} 5 &0 &0 &0 \\  0 &5 &0 &0  \end{bmatrix} &  {\bf R} &=diag\begin{bmatrix}0.5 & 0.5  \end{bmatrix} \\
	{\bf Q} &=diag\begin{bmatrix}0.25 & 0.25 & 0.25 & 0.25  \end{bmatrix}     &  {\bf U} &=diag\begin{bmatrix}2 & 2  \end{bmatrix}      \\
	{\bf W} &=diag\begin{bmatrix}5 & 5 & 1 & 1  \end{bmatrix}        &  {\bf Q}_a &=diag\begin{bmatrix}5 & 5  \end{bmatrix}     \\
	{\bf A}_a &=diag\begin{bmatrix}0.4 & 0.2 & 0.2 & 0.7  \end{bmatrix}           
\end{align*}

\bibliographystyle{IEEEtran}
\bibliography{IEEEabrv,Postdoc}

\end{document}